\documentclass[twoside]{irmaems}
\usepackage{amssymb} 
\usepackage{amsmath} 
\usepackage{latexsym}
\usepackage{graphicx} 

\renewcommand\theequation{\arabic{section}.\arabic{equation}}

\theoremstyle{definition} 

 \newtheorem{definition}{Definition}[section]
 \newtheorem{remark}[definition]{Remark}

\theoremstyle{plain}      

 \newtheorem{proposition}[definition]{Proposition}
 \newtheorem{theorem}[definition]{Theorem}
 \newtheorem{corollary}[definition]{Corollary}

\begin{document}

\title{Uniform distribution of 
Galois conjugates and  
beta-conjugates of a Parry number 
near the unit circle
and dichotomy of Perron numbers}

\author{
Jean-Louis Verger-Gaugry~
}

\address{
Institut Fourier, CNRS UMR 5582, \\
Universit\'e de Grenoble I, \\
BP 74, 38402 Saint-Martin d'H\`eres, France\\
email:\,\tt{jlverger@ujf-grenoble.fr}
}

\maketitle

\begin{abstract} 
Concentration and equi-distribution, 
near the unit circle, in Solomyak's set, of
the union of the Galois conjugates and the
beta-conjugates of 
a Parry number $\beta$ are characterized by 
means of the Erd\H{o}s-Tur\'an approach,  
and its improvements by Mignotte and Amoroso,
applied to the analytical function 
$f_{\beta}(z) = -1 + \sum_{i \geq 1} t_i z^i$
associated with 
the R\'enyi $\beta$-expansion $d_{\beta}(1)=
0.t_1 t_2 \ldots$ of unity. 
Mignotte's discrepancy function requires
the knowledge of the factorization of the Parry polynomial of $\beta$. This one
is investigated using theorems of Cassels, Dobrowolski,
Pinner and Vaaler, Smyth, Schinzel
in terms of cyclotomic, reciprocal non-cyclotomic and non-reciprocal factors.
An upper bound of Mignotte's discrepancy function
which arises from the beta-conjugates of $\beta$ which are
roots of cyclotomic factors is linked to the Riemann hypothesis, following Amoroso.
An equidistribution limit theorem, following Bilu's theorem, 
is formulated for the concentration phenomenon of conjugates of Parry numbers 
near the unit circle.
Parry numbers are Perron numbers. 
Open problems on non-Parry Perron numbers
are mentioned in the context of the existence of non-unique 
factorizations of 
elements of number fields into irreducible Perron numbers (Lind).
\end{abstract}



\newpage
\tableofcontents  

\section{Introduction}
\label{S1}

A Perron number is either 1 or
a real number $\beta > 1$
which is an algebraic integer
such that all its Galois conjugates
$\beta^{(i)}$ satisfy:
$|\beta^{(i)}| < \beta$ for all $i=1, 2, \ldots, d-1$, if
$\beta$ is of degree
$d \geq 1$ (with $\beta^{(0)} = \beta$).
Let 
$\mathbb{P}$ be the set of Perron numbers.
This set $\mathbb{P}$ is partitioned into
two disjoint components whose frontier is badly known
\cite{vergergaugry2}.
This partitioning 
arises from the properties of the
numeration in base $\beta$, i.e.
of the $\beta$-shifts and the dynamical systems  
$([0,1], T_{\beta})$ 
\cite{blanchard} \cite{frougny1}, 
where $\beta$ runs over
$(1, +\infty)$,
and where $T_{\beta}(x) = \{\beta x\}$
is the beta-transformation
($\lfloor x\rfloor$, $\{x\}$ and
$\lceil x \rceil$ denote the integer
part, resp. the fractional part, 
resp. the smallest integer greater than or equal to a real number $x$).  
Let us recall this dichotomy and fix some notations.

Let $\beta > 1$  
be a real number and assume throughout the paper
that 
$\beta$ is non-integer.
The R\'enyi $\beta$-expansion of $1$ is by definition
denoted by
\begin{equation}
\label{renyidef}
d_{\beta}(1) = 0.t_1 t_2 t_3 \ldots
\qquad
{\rm and ~corresponds~ to}
\qquad 1 = \sum_{i=1}^{+\infty} t_i \beta^{-i}\, ,
\end{equation}
where $t_1 = \lfloor \beta \rfloor,
t_2 = \lfloor \beta \{\beta\}\rfloor = \lfloor \beta T_{\beta}(1)\rfloor,
t_3 = \lfloor \beta \{\beta \{\beta\}\}\rfloor = \lfloor \beta T_{\beta}^{2}(1)\rfloor,
\ldots$ 
The digits
$t_i$ belong to the finite alphabet
$\mathcal{A}_{\beta} := \{0, 1, 2, \ldots, \lceil \beta - 1 \rceil \}$.  
$\beta$ is said to be a Parry number 
if $d_{\beta}(1)$ is finite or ultimately periodic (i.e. eventually
periodic); in particular, a Parry number $\beta$ is 
said to be simple if $d_{\beta}(1)$ is finite.

A proof that Parry numbers are 
Perron numbers 
is given by Theorem 7.2.13 and Proposition 7.2.21 
in Lothaire \cite{lothaire}. 
On the contrary a good proportion of Perron numbers are not Parry numbers, so that
the dichotomy
of $\mathbb{P}$ can be stated as
\begin{equation}
\label{dichotomy}
\mathbb{P} ~=~ \mathbb{P}_{P} ~\cup~ \mathbb{P}_{a},  
\end{equation}
where $\mathbb{P}_{P}$ denotes the set of Parry numbers and $\mathbb{P}_{a}$ the set of
Perron numbers which are not Parry numbers ($a$ stands for aperiodic). 
The set $\mathbb{P}_{P}$ is dense in $(1,+\infty)$ \cite{parry},  
contains all Pisot numbers \cite{bertrandmathis} \cite{boyd2} \cite{schmidt}
and Salem numbers of degree 4 \cite{boyd1}, at least \cite{boyd3} \cite{vergergaugry2}.

Following \cite{vergergaugry2} the present note 
continues the exploration
of this dichotomy by the Erd\H{o}s-Tur\'an approach applied to the collection
$\mathcal{B} := (f_{\beta}(z))_{\beta \in \mathbb{P}}$ of analytic functions 
\begin{equation}
\label{powerseriesdef}
f_{\beta}(z) := \sum_{i=0}^{+\infty} t_i z^i \,
\qquad \mbox{for}~ \beta \in \mathbb{P}, z \in \mathbb{C},
\end{equation}
with
$t_0 = -1$, canonically associated with the R\'enyi
expansions $d_{\beta}(1) = 0.t_1 t_2 t_3 \ldots$, 
for which $f_{\beta}(z)$ is a rational fraction 
if and only if $\beta \in \mathbb{P}_{P}$ (Section \ref{S2}); 
when $\beta \in \mathbb{P}_{P}$ the opposite of the 
reciprocal polynomial of the numerator of 
$f_{\beta}(z)$, namely $n_{\beta}^{*}(X)$,
is the Parry polynomial of $\beta$, a multiple of the
minimal polynomial of $\beta$.

Section \ref{S3} explores 
the factorization of $n_{\beta}^{*}(X)$, 
the geometry and the equi-distribution 
of its roots near the unit circle 
when
$\beta$ is a Parry number. 
These roots are
either Galois conjugates or 
beta-conjugates of $\beta$ and 
comparison is made between these two collections of roots.
The important points in this exploration are:
(i) the discrepancy function as obtained 
by Mignotte \cite{ganelius} 
\cite{mignotte1} \cite{mignotte2}
in a theorem which generalizes
Erd\H{o}s-Tur\'an's Theorem, its splitting and its properties
following Amoroso \cite{amoroso2} \cite{amorosomignotte},
how the beta-conjugates of $\beta$ which are roots of unity,
as roots of the irreducible cyclotomic factors   
of $n_{\beta}^{*}(X)$, are linked to 
the Riemann hypothesis (R.H.),
(ii) the number of 
positive real (Galois- or beta-)
conjugates of $\beta$ when the degree of $\beta$ is large,
by comparison with  Kac's formula
\cite{kac} \cite{edelmankostlan}, 
(iii) upper bounds for the multiplicities of beta-conjugates,  
(iv) an equidistribution limit theorem in the same 
formulation as Bilu's Theorem  
\cite{bilu} for convergent 
sequences of Parry numbers (or Parry polynomials)
with Haar measure on the unit circle as limit measure. 

Mignotte's discrepancy function allows a much better strategy 
than Erd\H{o}s-Tur\'an's discrepancy function in the Erd\H{o}s-Tur\'an
approach of the Parry polynomial $n_{\beta}^{*}(X)$
since it is a subaddtive function
on its factorizaton
and enables to investigate 
the roles played by 
its factors, namely cyclotomic, reciprocal non-cyclotomic and
non-reciprocal, term by term.

The factorization of the Parry polynomial $n_{\beta}^{*}(X)$
is itself 
a formidable challenge because of the difficulty of
determining the 
types of its factors, their multiplicities  
and the way   
it is correlated to the Rauzy fractal (central tile)
(Barat, Berth\'e, Liardet and Thuswaldner
\cite{baratbertheliardetthuswaldner}).
The problem of Lehmer of finding the smallest Mahler measure is essentially
equivalent to the problem of estimating the number of
irreducible non-cyclotomic factors of $n_{\beta}^{*}(X)$
(Pinner and Vaaler \cite{pinnervaaler3}).  

Section \ref{S4} provides various 
examples on which Erd\H{o}s-Tur\'an's and
Mignotte's discrepancy functions are
computed.  
In Section \ref{S5} 
the case of Perron numbers which 
are not Parry numbers is evoked in the general context
of the arithmetics of Perron numbers where  
non-unique factorizations on irreducible Perron numbers may
occur.  

\vspace{0.3cm}

Notations: N$(\beta) = \prod_{i=0}^{d-1} \, \beta^{(i)}$ 
is the algebraic norm of
the algebraic number $\beta (= \beta^{(0)})$, of degree $d \geq 1$;
$P_{\beta}(X)$ is the minimal 
polynomial of the algebraic number $\beta > 1$, 
with positive leading coefficient;   
$R^{*}(X)=X^m R(1/X)$ is the reciprocal polynomial of the
polynomial $R(X)$ (of degree $m$) and
$R(X)$ is said reciprocal if $R(X) = R^{*}(X)$; 
$||Q||_2 = \left(\sum_{j=0}^{m} |\alpha_j|^2\right)^{1/2}$, resp.
$||Q||_1 = {\rm L}(Q) = \sum_{j=0}^{m} |\alpha_j|$, 
${\rm H}(Q) = \max_{0 \leq j \leq m} |\alpha_j|$, 
the 2-norm, resp. the 1-norm (or length), 
resp. the height, of the polynomial
$Q(X) = \sum_{j=0}^m \alpha_j X^j$, $\alpha_j \in \mathbb{C}$; 
M$(R) = |a_R| \prod_{j=0}^{m} \max\{1, |b_{j}|\}$
denotes the Mahler measure of the polynomial 
$R(X) = a_R \prod_{j=0}^m (X - b_{j}) \in \mathbb{C}[X]$
where $a_R$ is the leading coefficient.
$D(z_0, r)$ denotes the open disk centred at 
$z_0 \in \mathbb{C}$ of radius $r > 0$,
$\overline{D(z_0, r)}$ its closure.
Log$^{+} \, x$ ($x > 0$) denotes 
max$\{{\rm Log} \, x, 0\}$.
The constants implied by the Vinogradov symbol `$\ll$'
are absolute and computable; when written
`$\ll_{\epsilon}$' for $\epsilon > 0$, they depend upon $\epsilon$.

\section{Szeg\H{o}'s Theorem and numeration}
\label{S2}

Every analytical function $f_{\beta}(z)$ 
with $\beta > 1$ ($\beta \not\in \mathbb{N}$) obeys
the dichotomy given by the following theorem
(\cite{szego1}, \cite{dienes} p 324--7) 
since its coefficients belong to
$\mathcal{A}_{\beta}$, which is finite.

\begin{theorem}[Szeg\H{o}]  
\label{szegothm}
A Taylor series $\sum_{n \geq 0} a_n z^n$ with
coefficients in a finite subset $S$ of $\mathbb{C}$ 
is either equal to
\begin {itemize}
\item[(i)] a rational fraction $\displaystyle
U(z) + z^{m+1} \frac{V(z)}{1-z^{p+1}}$
where $U(z) = -1 + \sum_{i=1}^{m} b_i z^i$ ,
$V(z) = \sum_{i=0}^{p} e_i z^i$ are
polynomials with
coefficients in $S$
and $m \geq 1, p \geq 0$
integers, or
\item[(ii)] it is an analytic function
defined on the open unit disk 
which is not continued beyond
the unit circle (which is its natural boundary).
\end{itemize}
\end{theorem}

Let us recall the Conditions of Parry \cite{blanchard} \cite{frougny1}
\cite{lothaire} \cite{parry}.
Let
$\beta > 1$ and $(c_i)_{i \geq 1}$ be the sequence of digits in 
$\mathcal{A}_{\beta}$ defined by
$$
c_1 c_2 c_3 \ldots := \left\{
\begin{array}{ll}
t_1 t_2 t_3 \ldots & \mbox{if ~$d_{\beta}(1) = 0.t_1 t_2 \ldots$~ is infinite,}\\
(t_1 t_2 \ldots t_{q-1} (t_q - 1))^{\omega}
& \mbox{if ~$d_{\beta}(1)$~ is finite,  
~$= 0. t_1 t_2 \ldots t_q$,} 
\end{array}
\right.
$$
where $( \, )^{\omega}$ means that the word within $(\, )$ is indefinitely repeated.
We say that $d_{\beta}(1)$ is finite if it ends in infinitely many zeros.
A sequence $(y_i)_{i \geq 0}$ of elements of 
$\mathcal{A}_{\beta}$ (finite or not) is said 
admissible if and only if 
$$(y_j, y_{j+1}, y_{j+2}, \ldots) <_{lex} 
~(c_1 c_2, c_3, \ldots) \qquad \mbox{for all}~ j \geq 0,$$ 
where $<_{lex}$ means {\it lexicographically smaller}.
A polynomial $\sum_{j=0}^{n} y_j X^j$, resp. a formal series
$\sum_{j=0}^{n} y_j X^j$, is said to be admissible
if and only if its   
coeffciients vector $(y_i)$ is admissible.
Since $f_{\beta}(X)+1$ is admissible by construction, the proof 
of Theorem \ref{szegothm} (i) by Dienes (\cite{dienes} pp 324--327)
readily shows
that both polynomials 
$U(X)+1$ and
$V(X)$ are admissible (we except the coefficient 
$t_0 = -1$ which does not belong to the alphabet $\mathcal{A}_{\beta}$)
when $\beta$ is a Parry number. 

In both cases, every analytical function $f_{\beta}(z)$
admits $1/\beta$ as simple root since $f_{\beta}'(1/\beta) =
\sum_{i \geq 1} i t_i \beta^{-i+1} > 0$. 

The integers $m \geq 1$ and $p+1 \geq 1$ in Theorem \ref{szegothm}
are respectively 
the preperiod length and the period length in
$d_{\beta}(1)$ when 
$\beta$ is a non-simple Parry number, and 
we take the natural convention $p+1=0$ 
when $\beta$ is a simple Parry number (which corresponds to
$V(X) \equiv 0$). The case $m=0$ corresponds to a purely periodic R\'enyi
expansion of unity in base $\beta$. 
Recall the $\beta$-transformation $T_{\beta}$:
$T_{\beta}^{1} = T_{\beta}: [0,1] \to [0,1],
~x \to \{\beta x\}$,
$T_{\beta}^{j+1}(x) = T_{\beta}(T_{\beta}^{j}(x)),
~j \geq 0$, and
$T_{\beta}^{0} = Id$. The sequence $(t_i)_{i \geq 1}$
is related to $(T_{\beta}^{j}(1))_{j \geq 0}$ \cite{frougny1} \cite{lothaire} by:
$$T_{\beta}^{0}(1)=1, ~T_{\beta}^{j}(1) = \beta^j - t_1 \beta^{j-1} - t_2 \beta^{j-2} - \ldots - t_j 
~\mbox{for}~ j \geq 1.$$ 
 
\section{Parry polynomials, Galois- and  beta-conjugates in
Solomyak's set $\Omega$ for a Parry number $\beta$}
\label{S3}

\subsection{Erd\H{o}s-Tur\'an's approach and Mignotte's discrepancy splitting} 
\label{S3.1}

Assume that $\beta > 1$ is a Parry number 
(of degree $d \geq 2$). Using
the notations of Theorem \ref{szegothm}, the rational fraction
$f_{\beta}(z)$ has 
$$U(z), \qquad \mbox{resp.} \quad U(z) (1-z^{p+1}) 
+ z^{m+1} V(z) ,$$
as numerator, when $\beta$ is a simple, resp. a non-simple,
Parry number 
whose coefficients are in $\mathbb{Z}$ and are of moduli 
in $\mathcal{A}_{\beta}$.
Since $f_{\beta}(1/\beta) = 0$, this numerator, say $- n_{\beta}(z)$,  
can be factored 
as
\begin{equation}
\label{numerat}
- n_{\beta}(z) = P_{\beta}^{*}(z) \times 
\prod_{j=1}^{s} \left(\Phi_{n_j}(z)\right)^{c_j} \times  
\prod_{j=1}^{q} \left(\kappa_{j}^{*}(z)\right)^{\gamma_j} \times
\prod_{j=1}^{u} \left(g_{j}(z)\right)^{\delta_j}
\end{equation}
where 
the polynomials $\Phi_{n_j}(X) \in \mathbb{Z}[X]$ are irreducible and cyclotomic, 
with $n_1 < n_2 < \ldots < n_s$, $\kappa_{j}(X) \in \mathbb{Z}[X]$
are irreducible and non-reciprocal,
$g_{j}(X) \in \mathbb{Z}[X]$ are irreducible, reciprocal and non-cyclotomic,
and
$s, q, u \geq 0$, $c_j, \gamma_j, \delta_j \geq 1$, are integers.
The Parry polynomial of $\beta$ \cite{boyd2} 
\cite{parry}, with 
leading coefficient $1$, is by definition
\begin{equation}
\label{numeratParryPol}
n_{\beta}^{*}(X):= P_{\beta}(X) \, 
\left( - \prod_{j=1}^{s} \left(\Phi_{n_j}(z)\right)^{c_j} \, 
\prod_{j=1}^{q} \left(\kappa_{j}(z)\right)^{\gamma_j} \,
\prod_{j=1}^{u} \left(g_{j}(z)\right)^{\delta_j}
\right)
\end{equation}
Its degree, denoted by 
$d_{P}$,
is $m+p+1$ with the notations of Theorem
\ref{szegothm}. 
In the case of non-simple Parry numbers, we have: 
$n_{\beta}^{*}(\beta) = 0 = T_{\beta}^{m+p+1}(1) - T_{\beta}^{m}(1)$ if $m \geq 1$
and $n_{\beta}^{*}(\beta) = 0 = T_{\beta}^{p+1}(1) - T_{\beta}^{0}(1) 
= T_{\beta}^{p+1}(1) - 1$ if $m=0$ (pure periodicity) where $p+1$ is the period length. 
Hence, for non-simple Parry numbers:
$$n_{\beta}^{*}(X) = 
X^{m+p+1} -t_1 X^{m+p} - t_2 X^{m+p-1} - \ldots - t_{m+p} X - t_{m+p+1}
\hspace{2cm} \mbox{}$$
\begin{equation}
\label{nonsimplepoly1}
\mbox{} \hspace{3cm}- X^{m} +t_1 X^{m-1} + t_2 X^{m-2} + \ldots + t_{m-1} X +t_m
\end{equation}
in the first case and 
\begin{equation}
\label{nonsimplepoly2}
n_{\beta}^{*}(X) = X^{p+1} -t_1 X^{p} - t_2 X^{p-1} - \ldots - t_{p} X - (1+t_{p+1})
\end{equation}
in the case of pure periodicity. For simple Parry numbers, the Parry polynomial
is
\begin{equation}
\label{simplepoly3}
X^{m} -t_1 X^{m-1} - t_2 X^{m-2} - \ldots - t_{m-1} X - t_{m}
\end{equation}
with $m \geq 1$.
Since $t_1 = \lfloor \beta \rfloor$, the Parry polynomial is a polynomial 
of small height which has the basic property: 
\begin{equation}
\label{hauteurineqs}
\lfloor \beta \rfloor \leq {\rm H}(n_{\beta}^{*}) 
\leq \lceil \beta \rceil
\end{equation}
with all coefficients having  a modulus 
$\leq \lfloor \beta \rfloor$ except possibly only one 
for which the modulus would be equal to $\lceil \beta \rceil$:
the coefficient of the monomial of degree $m$ in \eqref{nonsimplepoly1} and the
constant term 
in \eqref{nonsimplepoly2}.
If $\beta$ is a simple Parry number, then the equality
\begin{equation}
\label{hauteursimple}
{\rm H}(n_{\beta}^{*}) ~=~ \lfloor \beta \rfloor
\end{equation}
holds by \eqref{simplepoly3}. 
The polynomial 
\vspace{-0.3cm}
$$ \frac{n_{\beta}^{*}(X)}{P_{\beta}(X)}$$
is called the complementary factor \cite{boyd2}. 
It is a monic polynomial and some algebraic properties of
$n_{\beta}^{*}(X)$
are known (Handelman
\cite{handelman}).
A beta-conjugate of
$\beta$ is by definition the inverse $\xi$ of a zero of 
$f_{\beta}(z)$, or in other words a zero of the Parry polynomial, 
which is not a Galois conjugate of $\beta$
(i.e. $\xi \neq \beta^{(i)}$ for $i=1, 2, \ldots, d-1$).
Saying that the Parry polynomial of $\beta$ is irreducible is equivalent to saying that
$\beta$ has no beta-conjugate.

Let
$\mathcal{B} := \{f(z) = 1 + \sum_{j=1}^{\infty} a_j z^j \, \mid 0 \leq a_j \leq 1\}$ 
be the convex set of 
functions analytic in the open unit disk $D(0,1)$. Let 
$$\mathcal{G} := \{ \xi \in D(0,1) \mid f(\xi) = 0 ~\mbox{for some}~ f \in \mathcal{B}\}$$ 
and $\mathcal{G}^{-1} := \{ \xi^{-1} \mid \xi \in \mathcal{G} \}$.
The external boundary $\partial \mathcal{G}^{-1}$ of $\mathcal{G}^{-1}$ 
is a curve which has a cusp at $z=1$, a spike on the negative real axis, 
which is the segment $\bigl[-\frac{1+\sqrt{5}}{2}, -1\bigr]$,
and is fractal at an infinite number of points 
(Figure \ref{fractalSOLO};
see Fig. 2 and \S 4 in \cite{solomyak}). It defines
two components in the complex plane, and inverses of zeros of 
all $f \in \mathcal{B}$ are all necessarily within the bounded component
$\Omega := \mathcal{G}^{-1} \cup \overline{D(0,1)}$.

\begin{theorem}[Solomyak]
\label{solomyakthm}
The Galois conjugates $(\neq \beta)$ 
and the beta-conjugates of all Parry numbers  
$\beta$ belong to $\Omega$, occupy it densely, and
$$\mathbb{P}_{P} \cap \Omega = \emptyset.$$
\end{theorem}

\begin{proof}
By the following identity
\begin{equation}
\label{szegoshift}
f_{\beta}(z) ~=~ -1 + \sum_{i=1}^{\infty} t_i z^i ~=~
(-1 + \beta z)
\bigl(
1 + \sum_{j=1}^{\infty} T_{\beta}^{j}(1) z^{j}
\bigr),\qquad |z| < 1,
\end{equation}
the zeros $\neq \beta^{-1}$ of
$f_{\beta}(z)$ are
those of
$1 + \sum_{j=1}^{\infty} T_{\beta}^{j}(1) z^{j}$;  
but $1 + \sum_{j=1}^{\infty} T_{\beta}^{j}(1) z^{j}$ is 
a Taylor series which belongs to $\mathcal{B}$.
We deduce the claim.
\end{proof}

\begin{figure}
\begin{center}
\includegraphics[width=7cm]{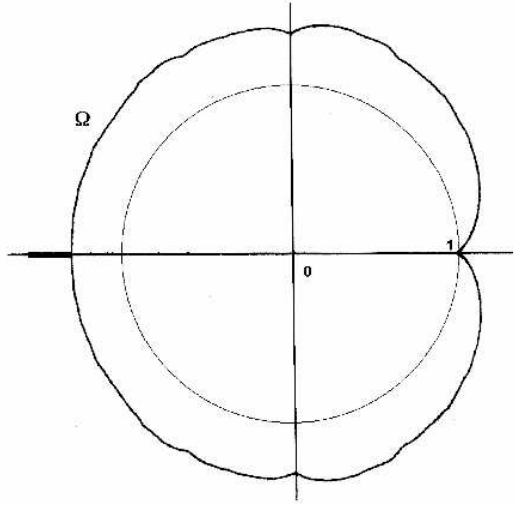}
\end{center}
\caption{Solomyak's set $\Omega$.} 
\label{fractalSOLO}
\end{figure}

Let us show that a phenomenon of high concentration and
equi-distribution of Galois conjugates 
$(\neq \beta)$ and beta-conjugates of a Parry number $\beta$
occur by clustering near the unit circle in $\Omega$: 
from a ``radial" viewpoint, using the 2-norm
$\|\cdot\|_2$ and the 1-norm 
$\|\cdot\|_1$ of the Parry polynomial
$n_{\beta}^{*}(X)$ (Theorem \ref{landauclustering}), 
and from an ``angular" viewpoint by
Mignotte's Theorem \ref{mignottethm} \cite{mignotte2}. 
Though densely distributed in $\Omega$ 
the conjugates 
of Parry numbers reach a very high concentration close to the unit circle
with maximality on the unit circle itself. Section \ref{S3.6} 
formulates limit theorems for this concentration phenomenon.

\begin{theorem}
\label{landauclustering}
Let $\beta > 1$ be a Parry number. 
Let $\epsilon > 0$ and $\mu_{\epsilon}$ the proportion of roots
of the Parry polynomial $n_{\beta}^{*}(X)$
of $\beta$, with $d_P = \deg(n_{\beta}^{*}(X)) \geq 1$, which lie
in $\Omega$ outside
the annulus $\left(\overline{D(0, (1-\epsilon)^{-1})} \setminus D(0,(1-\epsilon))
\right)$.
Then
\begin{equation}
\label{proportionroots}
(i) \quad \mu_{\epsilon} ~\leq~ 
\frac{2}{\epsilon \, d_P} \,
\left(
{\rm Log} \|n_{\beta}^{*}\|_2 - \frac{1}{2} {\rm Log}\beta  
\right),
\end{equation}
\begin{equation}
\label{proportionroots2}
(ii) \quad \mu_{\epsilon} ~\leq~
\frac{2}{\epsilon \, d_P} \,
\left(
{\rm Log} \|n_{\beta}^{*}\|_1 - \frac{1}{2} {\rm Log} \bigl|n_{\beta}^{*}(0)\bigr| 
\right).  
\end{equation}
\end{theorem}

\begin{proof}
(i) Let $\mu_1 d_P$ the number of roots of $n_{\beta}^{*}(X)$
outside $\overline{D(0, (1-\epsilon)^{-1})}$ in $\Omega$, except $\beta$
(since $\beta \not\in \Omega$).
By Landau's inequality \cite{landau}
$$
M(f) \leq \|f\|_{2} \qquad \quad \mbox{for}~ f(x) \in \mathbb{C}[X]$$
applied to $n_{\beta}^{*}(X)$ we deduce
$$\beta (1- \epsilon)^{-\mu_1 d_P} ~\leq~ 
M(n_{\beta}^{*}) ~\leq~ \|n_{\beta}^{*}\|_{2}.  
$$
Hence, since $-\mbox{Log}(1-\epsilon) \geq \epsilon$, 
$$
\mu_1 ~\leq~ 
\frac{1}{\epsilon} \, \left(
\frac{{\rm Log} \|n_{\beta}^{*}\|_{2}}{d_P} 
- \frac{{\rm Log} \beta}{d_P}
\right).$$
Let $\mu_2 d_P$ the
number of roots of $n_{\beta}^{*}(X)$
in $D(0, 1-\epsilon)$.
Then
$$(1-\epsilon)^{- \mu_{2} d_{P}} \leq M(n_{\beta}) \leq 
\|n_{\beta}\|_{2} = \|n_{\beta}^{*}\|_{2}  
$$
by Landau's inequality applied to $n_{\beta}(X)$.   
We deduce
$$\mu_2 ~\leq~
\frac{1}{\epsilon} \, 
\frac{{\rm Log} \|n_{\beta}^{*}\|_{2}}{d_P}
.$$
Since $\mu_{\epsilon} = \mu_1 + \mu_2$, we deduce the inequality \eqref{proportionroots}.

(ii) Applying Jensen's formula we deduce
\begin{equation}
\label{toto}
\frac{1}{2 \pi} \int_{0}^{2 \pi} {\rm Log} 
\bigl|n_{\beta}^{*}(e^{i \phi})\bigr|  d \phi
- {\rm Log} \bigl|n_{\beta}^{*}(0)\bigr|
= \sum_{|b_i| < 1} {\rm Log} \frac{1}{|b_i|}
\end{equation}
where $(b_i)$ is the collection of zeros of $n_{\beta}^{*}(z)$.
We have
$$\sum_{|b_i| < 1} {\rm Log} \frac{1}{|b_i|}
\geq 
\sum_{|b_i| < 1-\epsilon} {\rm Log} \frac{1}{|b_i|}
\geq~ \epsilon \, \mu_2 d_P .$$
From \eqref{toto}, since $\max_{\phi \in [0,2 \pi]} \, \bigl|n_{\beta}^{*}(e^{i \phi})\bigr|
\leq \|n_{\beta}^{*}\|_1$ , 
we deduce
$$
\mu_2 \leq~ \frac{1}{\epsilon \, d_P} \left(
{\rm Log} \|n_{\beta}^{*}\|_1
- {\rm Log} \bigl|n_{\beta}^{*}(0)\bigr|\right).$$
Now the roots of $n_{\beta}(z)$ inside $D(0,1-\epsilon)$
are the roots of $n_{\beta}^{*}(z)$ outside
the closed disk $\overline{D(0,(1-\epsilon)^{-1})}$, including
possibly $\beta$, so that their number is 
$\mu_1 d_P$ or $\mu_1 d_P + 1$.
Since $n_{\beta}^{*}(X)$ is monic, $| n_{\beta}(0) | = 1$.
We apply Jensen's formula to
$n_{\beta}(z)$ to deduce in a similar way
$$\mu_1 \leq~ \frac{1}{\epsilon \, d_P} 
{\rm Log} \|n_{\beta}\|_1 
$$
Since 
$\|n_{\beta}\|_1 = \|n_{\beta}^{*}\|_1$ 
and that $\mu_{\epsilon} = \mu_1 + \mu_2$,
we deduce \eqref{proportionroots2}.
\end{proof}

\begin{remark}
The terminology ``clustering near the unit circle" comes from the following fact:
if $(\beta_i)$ is a 
sequence of Parry numbers, of 
Parry polynomials of respective degree $d_{P,i}$ 
which satisfies
\begin{equation}
\label{seqinfinidegrees}
\lim_{i \to +\infty} d_{P,i} = +\infty 
\quad
{\rm and}
\quad
\lim_{i \to +\infty} \frac{{\rm Log} \, \beta_i}{d_{P,i}} = 0,
\end{equation}
then, since $\|n_{\beta_i}^{*}\|_{2} \leq
(d_{P,i} + 1)^{1/2} \,\lceil \beta_i \rceil$, 
the proportion
$\mu_{\epsilon,i}$ relative to $\beta_i$ satisfies
$$\mu_{\epsilon,i} \leq \frac{1}{\epsilon} \,
\left(
\frac{{\rm Log}(d_{P,i} + 1)}{d_{P,i}}
+
\frac{{\rm Log} \lceil \beta_i \rceil}{d_{P,i}}
\right)
$$
by \eqref{proportionroots}, what shows, for any real number $\epsilon > 0$, that 
\begin{equation}
\label{convergencespeed}
\mu_{\epsilon,i} \to 0, \qquad 
i \to +\infty.
\end{equation} 
The sufficient conditions 
\eqref{seqinfinidegrees}
for having convergence
of $(\mu_{\epsilon, i})_i$ to zero
already cover a large range of examples \cite{boyd3} \cite{vergergaugry2}.  
Let us notice that the conditions  
\eqref{seqinfinidegrees} do not imply that \begin{itemize}
\item
the corresponding
sequence $(d_i)_i$ of the degrees of the minimal polynomials
$P_{\beta_i}(X)$ tends to infinity; on the contrary, this sequence 
may remain bounded, even stationary (cf. 
the family of Bassino's cubic Pisot numbers in \cite{vergergaugry2}),
\item
the family of Parry numbers $(\beta_i)_i$ tends to infinity;
it may remain bounded or not 
(cf. Boyd's family of Pisot numbers
less than 2 in \cite{vergergaugry2}).
\end{itemize} 
\end{remark}

Define the radial operator
$\mbox{}^{(r)} : \mathbb{Z}[X] \to \mathbb{R}[X]$, 
$$R(X) = a_n \prod_{j=0}^{n} (X - b_j) ~\to~ 
R^{(r)}(X) = \prod_{j=0}^{n} \left(X-\frac{b_j}{|b_j|}\right).$$  
All the polynomials in the image of this operator have their roots
on the unit circle. This operator leaves invariant cyclotomic polynomials. It
has the property:
$P^{(r)} = (P^{*})^{(r)}$ for all polynomials $P(X) \in \mathbb{Z}[X]$
and 
is multiplicative:
$(P_{1} P_{2})^{(r)} = P_{1}^{(r)} P_{2}^{(r)}$ for 
$P_{1}(X), P_{2}(X) \in \mathbb{Z}[X]$.

The (angular) discrepancy relative to the 
distribution of the (Galois- and beta-)
conjugates of $\beta$ near the unit circle in
$\Omega$ is 
given by Erd\H{o}s-Tur\'an's Theorem \cite{erdosturan}
improved by Ganelius \cite{ganelius}, 
Mignotte 
\cite{mignotte2} and Amoroso \cite{amoroso2} \cite{amorosomignotte}, as follows.

\begin{theorem}[Mignotte]
\label{mignottethm}
Let $$R(X) = a_n X^n + a_{n-1} X^{n-1} + \ldots + a_1 X + a_0 = 
a_n \prod_{j=1}^{n} (X - \rho_j e^{i \phi_j}),$$  
$$a_n \neq 0, \qquad \rho_1, \rho_2, \ldots, \rho_n > 0,$$
be a polynomial with complex coefficients, where ~$\phi_j \in [0, 2 \pi)$~
for ~$j=1, \ldots, n$. For $0 \leq \alpha \leq \eta \leq 2 \pi$, put 
$N(\alpha, \eta) = $ Card$\{j \mid \phi_j \in [\alpha, \eta]\}$.
Let $k = \sum_{0}^{\infty} \frac{(-1)^{m-1}}{(2 m + 1)^2} = 0.916\ldots$ 
be Catalan's constant.
Then
\begin{equation}
\label{ineqdis}
\left|
\frac{1}{n} N(\alpha, \eta) - \frac{\eta - \alpha}{2 \pi} 
\right|^2 \leq 
\frac{2 \pi}{k} \times \frac{\tilde{h}(R)}{n}  
\end{equation}
where
\begin{equation}
\label{tildedis}
\tilde{h}(R) = \frac{1}{2 \pi} \int_{0}^{2 \pi} \, \mbox{{\rm Log}}^{+} |R^{(r)}(e^{i \theta})| d \theta
.
\end{equation}
\end{theorem}

Denote dis$(R) = \frac{\tilde{h}(R)}{n}$.
Let us call Mignotte's discrepancy function the rhs of \eqref{ineqdis} so that
$$C \cdot \mbox{dis}(R) ~=~ \frac{2 \pi}{k} \times \frac{\tilde{h}(R)}{n}  
$$
with $C = \frac{2 \pi}{k}= (2.619...)^2 = 6.859...$. 
The constant $C$, the same as in \cite{ganelius}, 
is much smaller than 16$^2 = 256$,  
computed in \cite{erdosturan}.  
Mignotte shows that $C$ cannot be less than
$(1.759...)^2 = 3.094...$ and that 
dis$(R)$ gives much smaller numerical estimates 
(cf. example in Section \ref{S4})  
than the expression 
$$\frac{1}{n} \, {\rm Log}\frac{L(R)}{\sqrt{|a_0 \, a_n|}}$$
proposed
by Erd\H{o}s-Tur\'an instead. 
In the following we investigate
Mignotte's discrepancy function (as in \cite{amoroso2}) 
but applied to the new class
of polynomials 
$$\{ R^{(r)}(X) \mid R(X) \in \mathcal{PP} \}$$
obtained by the radial operation $^{(r)}$ 
from the set of Parry polynomials
$$\mathcal{PP} := \{ n_{\beta}^{*}(X) \mid \beta \in \mathbb{P}_{P} \}.$$

The angular control of the geometry of the beta-conjugates
of a Parry number $\beta$ with respect to the 
geometry of its Galois-conjugates, by 
rotating sectors of suitable opening angles
in the unit disk \cite{vergergaugry2}, 
is best for smallest possible estimates of the discrepancy function.

Mignotte's discrepancy function on 
$\mathbb{C}[X] \setminus \{0\}$ never 
takes the value zero for the simple reason that,
in \eqref{ineqdis}, the function 
$(\alpha, \eta) \to (\eta - \alpha)/2 \pi$ is continuous and
that $(\alpha, \eta) \to N(\alpha, \eta)/n$ 
takes discrete values, so that
\begin{equation}
\label{minipopo}
{\rm dis} \bigl(n_{\beta}^{* \, (r)}\bigr) ~>~ 0 
\end{equation}
for all $\beta \in \mathbb{P}_{P}$.  
In Section \ref{S3.5} we  
give a lower estimate of 
${\rm dis} \bigl(n_{\beta}^{* \, (r)}\bigr)$.

Schinzel \cite{schinzel} has asked a certain number of questions, and
reported some conjectures and theorems, 
on the number of cyclotomic factors, resp. non-cyclotomic non-reciprocal,
resp. non-cyclotomic reciprocal factors in the 
factorization into irreducible factors 
of a general polynomial
with integer coefficients. 
For giving an upper estimate of 
${\rm dis} \bigl(n_{\beta}^{* \, (r)}\bigr)$ 
we will reformulate these  questions 
in the particular context of Parry polynomials 
and state some theorems (Section \ref{S3.2}).    
Then, as a consequence, 
we will  
separate out the contributions relative to 
the irreducible factors, since
Mignotte's discrepancy function allows to do it, by the properties
of Log$^{+} x$ and the multiplicativity
of the radial operator $^{(r)}$: 
from \eqref{numerat} and \eqref{tildedis},  
with $R(X) = n_{\beta}(X)$ or
$R(X) =  n_{\beta}^{*}(X)$, 
the splitting is as follows
\begin{equation}
\label{decompodis}
\tilde{h}(n_{\beta}^{*}) ~=~ \tilde{h}(n_{\beta}) ~\leq~ 
\tilde{h}(P_{\beta}) + 
\tilde{h}(\prod_{j=0}^{s} \, \Phi_{n_j}^{c_j}) +
\tilde{h}(\prod_{j=0}^{q} \, \kappa_{j}^{\gamma_j}) 
+ 
\tilde{h}(\prod_{j=0}^{u} \, g_{j}^{\delta_j}).  
\end{equation}
Section \ref{S3.3} is relative to the contributions 
of cyclotomic factors and Section \ref{S3.4} 
to those of non-cyclotomic factors.   
In Section \ref{S4} 
the numerical optimality 
of this splitting process is studied on examples.

\subsection{Factorization and irreducible factors}
\label{S3.2}

Let $\beta > 1$ be a Perron number 
given by its minimal polynomial $P_{\beta}(X)$, 
for which we
know that it is a Parry number. 
The factorization of its Parry polynomial
$n_{\beta}^{*}(X)$  
amounts to the knowledge of: 
\begin{itemize}
\item[(i)] its degree $d_P$, as a function of $\beta$, 
its Galois conjugates and $d$ the degree of the minimal
polynomial $P_{\beta}(X)$,
\item[(ii)] its distinct irreducible factors and their multiplicities.  
\end{itemize}

The determination of degree $d_P = m+p+1$ of $n_{\beta}^{*}(X)$
(with the notations
of Theorem \ref{szegothm}), which expresses the dynamics of the $\beta$-transformation,
brings to light the need of the geometrical representation of the
$\beta$-shift, the Rauzy fractal \cite{pytheasfogg}
\cite{baratbertheliardetthuswaldner}, 
above the set $\mathbb{Z}_{\beta}$
of $\beta$-integers,
for providing
upper estimates of 
it, as shown below for Pisot numbers.
For a generic Parry number $\beta$, an explicit formula for 
$d_P$ as a function of $\beta$, 
its Galois conjugates $\beta^{(i)}$
and the degree $d$ of
the minimal polynomial $P_{\beta}(X)$ is probably 
difficult to obtain in general, if it exists.

Assume that $\beta > 1$ is a Pisot number, of degree $d \geq 2$.
We refer to \cite{gazeauvergergaugry} p. 142 and Lemma 4.6 
for details.
Let $\{Z_0, Z_1, \ldots, Z_{d-1}\}$ be the canonical basis of
$\mathbb{R}^{d}$ and $x \cdot y = \sum_{j=0}^{d-1} x_j y_j$ the standard scalar
product in this basis.
Let $B = \mbox{}^{t}(1 ~\beta~ \beta^2 \ldots \beta^{d-1})$  and
$u_B = \|B\|^{-1} B$, 
where $\mbox{}^{t}$ means transposition. Let 
$\mbox{}^{t}Q$ be the companion matrix of $\beta$, 
$\pi_B$ is the orthogonal projection mapping onto
$\mathbb{R} u_B$. \,$F$ denotes   
generically 
a $\mbox{}^{t}Q$-invariant subspace of $\mathbb{R}^{d}$, of
dimension 1 if the eigenvalue of
$\mbox{}^{t}Q$ on $F$ (i.e. one of the Galois conjugates
$\beta^{(i)}$) is real, resp. of dimension 2 if  
it is complex, except the expanding line $\mathbb{R} u$
(\cite{gazeauvergergaugry} Theorem 3.1
and Theorem 3.5). The projection mapping onto $F$,
$\mathbb{R}^{d} \to F$, along its 
$\mbox{}^{t}Q$-invariant 
complementary space
is denoted by $\pi_F$.
Let $p_2 = \oplus_{F} \pi_F$
be the sum of the projection mappings.  
Let $\mathcal{C}' := \left\{\sum_{j=0}^{d-1} \alpha_j Z_j \mid
\alpha_j \in [-1, +1] \, \mbox{for all}~ j=0, 1, \ldots, d-1\right\}$,
$\delta_{F}' := \max_{x \in \mathcal{C}'} \|\pi_{F}(x)\|$, 
$\lambda_F :=$ the absolute value of the Galois conjugate of
$\beta$ associated with $F$, and
$c_{F}' := \lfloor \beta \rfloor \frac{\delta_{F}'}{1 - \lambda_{F}^{d}}$.  
The canonical cut-and-project scheme associated with
the set $\mathbb{Z}_{\beta}$ of the beta-integers
is
$$\mathbb{R} u_B ~\stackrel{\pi_B}{\longleftarrow}~ (\mathbb{R} u_B \times D \simeq \mathbb{R}^{d}, \mathbb{Z}^{d})
~\stackrel{p_2}{\longrightarrow}~ D = \oplus F.$$
Let
$\Omega'$ be 
the $(d-1)$-dimensional window 
$\oplus_F \Omega'_{F}$ in $D$, direct sum of 
the windows $\Omega'_{F}$ on $F$ 
defined by
$$
\Omega'_{F}
=
\left\{
\begin{array}{l}
\mbox{closed interval centred at 0 in $F$ of length $2 c_{F}'$ if dim $F = 1$},\\
\mbox{closed disk centred at 0 in $F$ of radius $c_{F}'$ if dim $F = 2$}.
\end{array}
\right.
$$
This window
$\Omega'$ is sized at its best on each $1$-dim or $2$-dim
subspace $F$ to contain the $1$-dimensional, resp. the $2$-dimensional, 
sections of the 
central tile of the Rauzy fractal. 

\begin{theorem}
\label{dPupperboundthm}
Let $\beta > 1$ be a Pisot number of degree $d \geq 2$.
Then
\begin{equation}
\label{dpineq}
d_P
\leq
~\# 
\left\{x \in \mathbb{Z}^d \mid p_{2}(x) ~\in \,
\frac{{\rm H}(n_{\beta}^{*})}{\lfloor \beta \rfloor} \, \Omega' , 
~~\pi_{B}(x) \cdot u_B ~\in \Bigl[0, \frac{1}{\|B\|}\Bigr) \right\}. 
\end{equation}
\end{theorem}
 
\begin{proof}
The polynomials $T_{\beta}^{j}(1)$, for $j=1, 2, \ldots, d_P$,  
are all polynomials in $\beta$ with coefficients
having their modulus in  $\{0, 1, \ldots, {\rm H}(n_{\beta}^{*})\}$, 
of degree $j$, which are equal 
to their fractional part. The collection $(T_{\beta}^{j}(1))_{j=1, \ldots, d_P - 1}$ 
is a family of linearly independent
polynomials. Theorem \ref{dPupperboundthm}
is then a consequence of Lemma 4.6 in \cite{gazeauvergergaugry}.
\end{proof}

\begin{remark}
A much better upper bound of $d_P$ would be given by the rhs of
\eqref{dpineq} in which the ``box" 
$\Omega'$ is replaced by the true central tile of the Rauzy fractal
\cite{pytheasfogg}.
Indeed, this central tile may be disconnected, may contain lots of holes, and its topology
is a prominent ingredient for counting points of the lattice
$\mathbb{Z}^{d}$ which are projected by $p_2$ to
this central tile.   
\end{remark}

Let us now turn to the factorization of 
$n_{\beta}^{*}(X)$ for $\beta$ a Parry number.

\begin{theorem}
\label{multiplicitybetaconju}
Let $\beta$ be a Parry number. If $\xi$ is a beta-conjugate of $\beta$
which is not a unit, then its multiplicity $\nu_{\xi}$ as root of the Parry polynomial
$n_{\beta}^{*}(X)$ satisfies:
\begin{equation}
\label{multicxi}
\nu_{\xi} ~\leq~ \frac{1}{\log 2} \Bigl( \log\bigl({\rm H}(n_{\beta}^{*})\bigr) - \log |N(\beta)| \Bigr).
\end{equation}
Moreover, if 
\begin{equation}
\label{normabove}
|N(\beta)| ~\geq~ \frac{{\rm H}(n_{\beta}^{*})}{3}, 
\end{equation}
then all beta-conjugates of $\beta$ which are not units (if any)
are simple roots of
$n_{\beta}^{*}(X)$.
\end{theorem}

\begin{proof}
From \eqref{numerat}, since $P_{\beta}(X)$ divides $n_{\beta}^{*}(X)$
and that H$(n_{\beta}^{*}) \in \{\lfloor \beta \rfloor, \lceil \beta \rceil\}$
(cf. \eqref{hauteurineqs}, \eqref{hauteursimple}), we have
$$
\bigl|
\prod_{j=1}^{q} \left(\kappa_{j}^{*}(0)\right)^{\gamma_j} \bigr| \times
\bigl| \prod_{j=1}^{u} \left(g_{j}(0)\right)^{\delta_j} \bigr|
~\leq~ \frac{{\rm H}(n_{\beta}^{*})}{|N(\beta)|}.$$
As a consequence, if $\xi$ is a beta-conjugate which is not 
a unit, then, since the irreducible
factors of $n_{\beta}^{*}(X)$ are all monic, the inequality $|N(\xi)| \geq 2$ implies
$$2^{\nu_{\xi}} ~\leq~ \frac{{\rm H}(n_{\beta}^{*})}{|N(\beta)|}.$$
Hence the claim. 
Now, if
$|N(\beta)| ~\geq~ \frac{{\rm H}(n_{\beta}^{*})}{3}$ then 
$$\bigl|
\prod_{j=1}^{q} \left(\kappa_{j}^{*}(0)\right)^{\gamma_j} \bigr| \times
\bigl| \prod_{j=1}^{u} \left(g_{j}(0)\right)^{\delta_j} \bigr|
 \leq 3,$$
which necessarily implies that
$\nu_{\xi} = 1$ for each beta-conjugate $\xi$ of $\beta$ which is not a unit. 
\end{proof}

\begin{corollary}
\label{betaconj13simple}
The beta-conjugates of a Parry number $\beta \in (1, 3)$
which are not units
are always simple roots of the Parry polynomial of $\beta$.
\end{corollary}

\begin{proof}
Indeed, ${\rm H}(n_{\beta}^{*}) \in \{\lfloor \beta \rfloor,
\lceil \beta \rceil\}$ and $\lceil \beta \rceil \leq 3$. Thus
${\rm H}(n_{\beta}^{*})/3 \leq 1$. But 
$|N(\beta)| \geq 1$ so that
\eqref{normabove} is satisfied.  
\end{proof}




Let $\beta$ be a Parry number, for which the Parry polynomial
$n_{\beta}^{*}(X)$ is factored as in \eqref{numeratParryPol}. We have

$$
\begin{array}{cl}
1+s+q+u ~= & 
\mbox{number of its distinct irreducible factors}\\
\sum_{j=1}^{s} c_j ~= & \mbox{number of its cyclotomic (irreducible)}\\
 & \mbox{factors counted with multiplicities}\\ 
1+\sum_{j=1}^{s} c_j + \sum_{j=1}^{q} \gamma_j + \sum_{j=1}^{u} \delta_j ~= & 
\mbox{number of its irreducible factors}\\
& \mbox{counted with multiplicities}\\
1+ \sum_{j=1}^{q} \gamma_j + \sum_{j=1}^{u} \delta_j ~= & 
\mbox{number of its non-cyclotomic irreducible}\\
& \mbox{factors counted with multiplicities}\\
1+q+u ~= & \mbox{number of its non-cyclotomic irreducible}\\
& \mbox{factors counted without multiplicities}\\
\gamma + \sum_{j=1}^{q} \gamma_j ~= & 
\mbox{number of its non-reciprocal irreducible}\\
& \mbox{factors counted with multiplicites, with}\\
& \mbox{$\gamma = 1$ if $P_{\beta}(X)$ is non-reciprocal, and}\\
& \mbox{$\gamma = 0$ if $P_{\beta}(X)$ is reciprocal}\\
\end{array}  
$$

The remarkable result of Smyth \cite{smyth} implies easily
\begin{theorem}
\label{smyththm}
For every Parry number $\beta$, the inequality
\begin{equation}
\label{nonreciprocalineq} 
\gamma + \sum_{j=1}^{q} \gamma_j ~<~ 
\frac{\log \|n_{\beta}^{*}\|_2}{\log \theta_0}
\end{equation}
holds where $\theta_0 = 1.3247...$ is 
the smallest Pisot number, dominant root of $X^3-X-1$,
where $\gamma = 1$ if $P_{\beta}(X)$ is non-reciprocal and
$\gamma = 0$ if $P_{\beta}(X)$ is reciprocal.
\end{theorem}

Let us remark that the upper bound in \eqref{nonreciprocalineq}
for the number of non-reciprocal irreducible factors of the Parry polynomial of $\beta$
depends upon its $2$-norm,  not of the degree $d_P$,
meaning that it is strongly dependent upon the
gappiness (lacunarity) of the R\'enyi $\beta$-expansion $d_{\beta}(1)$ of unity
\cite{vergergaugry1},
a strong gappiness leading to a small number of non-reciprocal factors
in $n_{\beta}^{*}(X)$.

\begin{corollary}
\label{lacunaritynonreciprocal}
If $\beta$ is a Parry number for which 
the minimal polynomial is non-reciprocal
and $d_{\beta}(1)= 0. t_1 t_2 t_3 \ldots$, 
of preperiod length $m \geq 0$ and period length $p+1$,
satisfies (with $t_0 = -1$)
\begin{equation}
\label{lacupisotineq}
\left.
\begin{array}{ll}
\mbox{if $\beta$ is simple} & \displaystyle
\sum_{j=0}^{m} t_{j}^{2}  
\\
\mbox{if $\beta$ is non-simple} &
\displaystyle
\sum_{j=0}^{p} t_{j}^{2} + (1+t_{p+1})^2 + \sum_{j=1}^{m} (t_j - t_{p+j+1})^2   
\end{array}
\right\}
\leq \theta_{0}^{4} = 3.0794...
\end{equation}
then   
the Parry polynomial of $\beta$ has no non-reciprocal irreducible factor in it.
\end{corollary}

\begin{proof}
Indeed, in this case, $\gamma = 1$; 
\eqref{lacupisotineq} is equivalent to
$$
\frac{\log \|n_{\beta}^{*}\|_2}{\log \theta_0} \leq 2
\quad \Leftrightarrow \quad \|n_{\beta}^{*}\|_{2}^{2} ~\leq~ \theta_{0}^{4}$$
since
$\|n_{\beta}^{*}\|_{2}^{2}$ is
given by
\eqref{nonsimplepoly1}, \eqref{nonsimplepoly2}
and \eqref{simplepoly3}.
Therefore, from \eqref{nonreciprocalineq}, we deduce 
$\sum_{j=1}^{q} \gamma_j = 0$ what implies
the claim.
\end{proof}

Let us make explicit 
\eqref{lacupisotineq}
in the ``simple" case.  
We obtain Parry numbers 
$\beta$ for which
$d_{\beta}(1)$ has necessarily the form
$$d_{\beta}(1) = 0 . 1 \underbrace{0 0 \ldots 0}_{\delta} 1$$
for some integer $\delta \geq 0$, i.e. $\beta = \beta_{\delta}$ is the 
dominant root of
the polynomial $X^{\delta+2} - X^{\delta+1} -1$. The
algebraic integers $(\beta_{\delta})_{\delta \geq 3}$ are Perron numbers studied
by Selmer \cite{vergergaugry2}. 
The case $\delta = 0$ corresponds to
the golden mean $\tau = (1+\sqrt{5})/2$
since $d_{\tau}(1)=0.11$.  

\vspace{0.5cm}

The special sequence $(\Phi_{n_j})_{j=1,\ldots,s}$ 
of cyclotomic polynomials in the factorization of
$n_{\beta}^{*}(X)$ 
is such that $\sum_{j=1}^{s} c_j \varphi(n_j) \leq d_P - d$, with
$s \leq n_{s}$, where $\varphi(n)$ is the Euler function,
and its determination is complemented by
the remarkable Theorem 3 of Schinzel \cite{schinzel}  
which readily leads to

\begin{theorem}
\label{schinzelthm}
There exists an absolute constant $C_0 > 0$ such that, 
for every Parry number $\beta$, the 
number $s$ of distinct cyclotomic irreducible factors of the Parry polynomial
of $\beta$ satisfies
\begin{equation}
\label{cyclotomicineq}
s ~\leq~ C_0 \sqrt{d_P}.  
\end{equation}
\end{theorem} 

Concerning the non-cyclotomic factors of $n_{\beta}^{*}(X)$, the 
remarkable Theorem 2 of Dobrowolski \cite{dobrowolski}
implies 

\begin{theorem}
\label{dobrowolskithm}
There exists an absolute  constant $C_1 > 0$ such that
for every Parry number $\beta$ and $\epsilon > 0$ an arbitrary positive real number, then
\begin{equation}
\label{noncyclotomicineq}
1+ \sum_{j=1}^{q} \gamma_j + \sum_{j=1}^{u} \delta_j
~\leq ~ C_1 \, \Bigl(
\, (\frac{d_P}{\log \|n_{\beta}^{*}\|_{2}^{2}})^{\epsilon} \times \log \|n_{\beta}^{*}\|_{2}
\, \Bigr).
\end{equation}
\end{theorem}

Both constants $C_0 , C_1$ in \eqref{cyclotomicineq} and
in \eqref{noncyclotomicineq} are specific 
to the whole collection of Parry polynomials
$\mathcal{P}\mathcal{P}$; they are almost surely different from       
the constants relative to general polynomials with integral coefficients
as in \cite{schinzel} and \cite{dobrowolski}
and are expected to be computable when 
$\mathcal{P}\mathcal{P}$ will be characterized.
  
\subsection{Cyclotomic factors}
\label{S3.3}

This section is relative to 
the beta-conjugates $\chi$ 
of $\beta$ for which the minimal polynomial of $\chi$ is cyclotomic.
The M\"obius function $\mu$ is
given by, for $n \geq 1$,
$$\mu(n) ~=
\left\{
\begin{array}{cc}
0 & \mbox{if $n$ is divisible by a square,}\\
(-1)^{k} & \mbox{if $n$ is the product of $k$ prime numbers.}
\end{array}
\right.
$$
For $n \geq 1$, 
$$X^n - 1 = \prod_{j=1}^{n} \Phi_{j}(X).$$
By the M\"obius inversion formula,
the $n$-th cyclotomic polynomial is
$$\Phi_{n}(X) := \prod_{j=1}^{n} (X^d - 1)^{\mu(n/j)}.$$  

Amoroso in \cite{amoroso1} proves the remarkable result
that the 
assertion that the Riemann zeta function does not vanish for
Re$ z \geq \sigma + \epsilon$ is equivalent to the
inequality $\tilde{h}\Bigl(\prod_{n=1}^{N} \Phi_{n}
\Bigr)~\ll~N^{\sigma +\epsilon}$, 
where $\sigma$ is the supremum of the real parts
of the non-trivial zeros of the Riemann zeta function, 
and where $\sigma = 1/2$ if the Riemann hypothesis (R.H.) 
is true. Here we  
reconsider the same arguments to adapt them to the particular 
products of cyclotomic
polynomials which appear in the factorization of the 
Parry polynomial
$n_{\beta}^{*}(X)$ as in \eqref{numeratParryPol}. 

\begin{theorem}
\label{cyclothm}
Let $s \geq 1$. Let  
$\tilde{c} = (c_j)_{j = 1, \ldots, s}$
be a collection of positive integers
and $n_1 \leq n_2 \leq \ldots \leq n_s$ be 
an increasing sequence of positive integers.
Let $\tilde{n} = (n_j)_{j=1,\ldots,s}$.
Assume that the Riemann hypothesis is true.
Then
there exists a linear form
$L_{\tilde{n}}(\tilde{c})$ with rational coefficients
such that
\begin{equation}
\label{contribcyclopoly}
\tilde{h}\Bigl(\, \prod_{j=1}^{s} \, \Phi_{n_j}(X)^{c_j}\Bigr) 
~\ll_{\epsilon}~ \Bigl|L_{\tilde{n}}(\tilde{c})\Bigr| \, 
n_s^{\mbox{}~\frac{1}{2} + \epsilon}.   
\end{equation}
\end{theorem}

\begin{proof}
Let $N = n_s$. Let
$$G(X) = \prod_{n=1}^{N} \Phi_{n}(x)^{\sigma_{n}}$$
with
$$\sigma_n ~=~
\left\{
\begin{array}{ll}
0 & \mbox{if}~ n \not\in \{n_1, n_2, \ldots, n_s\}\\
c_j & \mbox{if}~ n= n_j ~\mbox{for}~ j \in \{1, 2, \ldots, s\}
\end{array}
\right.
$$
for $n \geq 0$.
The remarkable Theorem 4.2 of Amoroso in \cite{amoroso2} implies
\begin{equation}
\label{hprodCYCLO}
\tilde{h}(G) ~\leq~
\sqrt{\frac{\pi}{12} \sum_{m=1}^{N} \, \left( \sum_{j | m} \frac{\mu(j)}{j^2}\right) 
\left(\sum_{n \leq N/m} \, \sigma_{m n} \sum_{ k | n} \frac{\mu(k) k}{n}\right)^2}
\end{equation}

Since $0 \leq \sum_{j | m} \frac{\mu(j)}{j^2} \leq 1$
and that
$$
\sum_{n \leq N/m} \, \sigma_{m n} \sum_{ k | n} \frac{\mu(k) k}{n}
\quad \mbox{can be written} \quad
\sum_{k=1}^{N/m} L_{k,m}(\tilde{c}) \mu(k)
$$
where $(L_{k,m}(\tilde{c}))$ are a family of linear forms with
rational coefficients,
we deduce
$$
\sum_{m=1}^{N} \, \left( \sum_{j | m} \frac{\mu(j)}{j^2}\right)
\left(\sum_{n \leq N/m} \, \sigma_{m n} \sum_{ k | n} \frac{\mu(k) k}{n}\right)^2
~\leq~
\sum_{m=1}^{N} \left|L_{m}(\tilde{c})\right|^2
\left(\sum_{k=1}^{N/m} \mu(k)\right)^2$$
for some linear forms 
$L_{m}(\tilde{c})$ deduced from  the family
$(L_{k,m}(\tilde{c}))$ according to the signs of
$\mu(k)$. Let us take 
$L_{\tilde{n}}(\tilde{c})$ such that
$\left|L_{\tilde{n}}(\tilde{c})\right| = \sup_{m=1, \ldots,N} 
\left|L_{m}(\tilde{c})\right|$.  
The Riemann hypothesis is equivalent
(Titchmarsh \cite{titchmarsh} 14.25C) to
$$\sum_{k \leq x} \mu(k) ~\ll~ x^{\frac{1}{2} + \epsilon} ~~\mbox{for all real $x$ and all
$\epsilon > 0$}.$$
Then there exists a constant $A > 0$ such that, for $\epsilon > 0$,  
$$\tilde{h}(G)^2 ~\leq~ 
\frac{\pi}{12} \, A \, \sum_{m=1}^{N} \left(\frac{N}{m}\right)^{2 \, (\frac{1}{2}+\epsilon)}
\times
\left|L_{\tilde{n}}(\tilde{c})\right|^{2} ~\leq~ 
\frac{\pi}{12} \, A \, 
\left|L_{\tilde{n}}(\tilde{c})\right|^{2} \, N^{1 + 2 \epsilon} \,
\sum_{m=1}^{\infty} \frac{1}{m^{1 + 2 \epsilon}}.$$
We deduce \eqref{contribcyclopoly}.  
\end{proof}

Let us turn to making explicit
upper bounds of the multiplicities of the primitive 
roots of unity involved in the product
$\prod_{j=1}^{s} \, \Phi_{n_j}(X)^{c_j}$. 
Let $n \geq 1$ and $\zeta_n$ be a primitive $n$th-root of unity.
Let us write 
the factorization 
of $\Phi_{n}(X)$
in $\mathbb{Q}(\zeta_n)[X]$
as $$\Phi_{n}(X) = \prod_{m=1}^{\varphi(n)} \Phi_{n,m}(X).$$   
The polynomial $\Phi_{n,m}(X)$ is $X - \xi_m$ for some
primitive $n$th-root $\xi_m$ of unity.
Then
\begin{equation}
\label{producyclo}
\prod_{j=1}^{s} \, \Phi_{n_j}(X)^{c_j}
~=~
\prod_{n=1}^{\infty} \prod_{m=1}^{\varphi(n)} \Phi_{n,m}(X)^{e(n,m)}
\end{equation}
for some integers $e(n,m) \geq 0$. 
The total number of cyclotomic factors of
the Parry polynomial
$n_{\beta}^{*}(X)$ is then
$$\sum_{j=1}^{s} c_j \varphi(n_j) ~=~
\sum_{n=1}^{\infty} \sum_{m=1}^{\varphi(n)} e(n,m).$$

Let $$\alpha(\mathbb{Q}) := 
\lim_{X \to \infty} \frac{1}{X} \sum_{\stackrel{n=1}{\varphi(n) \leq X}}^{+\infty} \varphi(n).$$  
We have \cite{pinnervaaler1}: $\alpha(\mathbb{Q}) =
\frac{\zeta(2) \zeta(3)}{\zeta(6)}$ 
where $\zeta(z)$ is the Riemann zeta function.

At each place $v$ of the number field $\mathbb{K}$ 
we write $\mathbb{K}_v$ for
the completion of  $\mathbb{K}$ at $v$, 
$\overline{\mathbb{K}}$
for an algebraic closure of $\mathbb{K}_v$ and
$\Omega_v$ for the completion of $\mathbb{K}_v$.
The field $\Omega_v$ is complete as a metric space and
algebraically closed.
Two absolute values $|~|_v$ and $\|~\|_v$ are introduced on 
$\Omega_v$. If $v | \infty$ then $\|~\|_v$ restricted to
$\mathbb{Q}$ is the usual Archimedean absolute value. If
$p$ is a prime number and $v|p$ then
$\|~\|_v$ restricted to
$\mathbb{Q}$ is the usual $p$-adic value.
They are related by
$$|~|_v := \|~\|_{v}^{[\mathbb{K}_v : \mathbb{Q}_v ]/[ \mathbb{K} : \mathbb{Q} ]}.
$$
Now let $F(X) \in \mathbb{Q}[X]$ and introduce 
the global measure of $F$ as
$$\nu(F) := \prod_v \, \nu_{v}(F)$$ 
where
$$\nu_{v}(F) := \sup\{|F(z)|_v ~\mid~ z \in \Omega_v ~\mbox{and}~ |z|_v =1 \}.$$ 
We can now introduce
$$\mathcal{R} := 
\max\left\{
\frac{d_P}{{\rm Log} \, \nu(n_{\beta}^{*})}, \, 3 \right\}.$$
This quantity plays an important role in the multiplicites of
the cyclotomic factors by Theorem \ref{pinnervaalerthm1}. 
Its inverse $\mathcal{R}^{-1}$ is surprisingly
deeply correlated
to the convergence condition  
\eqref{seqinfinidegrees} relative to convergent
sequences of Parry numbers by the following
inequalities (Pinner and Vaaler \cite{pinnervaaler1}, Lemma 2):
\begin{equation}
{\rm Log \, H}(n_{\beta}^{*}) ~\leq~
{\rm Log} \, \nu(n_{\beta}^{*})  
~\leq~ 
2 {\rm Log \, H}(n_{\beta}^{*}). 
\end{equation} 
Indeed, since H$(n_{\beta}^{*}) \in \{\lfloor \beta \rfloor, \lceil \beta \rceil\}$
we have the following estimate of the global measure of
the Parry polynomial of $\beta$:
$$\nu(n_{\beta}^{*}) ~\in~ [\lfloor \beta \rfloor, \lceil \beta \rceil^2 ]$$
and $\mathcal{R}^{-1}$ is roughly equal to Log$\, \beta/d_P$
when $d_P$ is large enough compared to Log H$(n_{\beta}^{*})$.

The remarkable Theorem 1 of Pinner and Vaaler \cite{pinnervaaler1} 
gives a system of four inequalities  
for sums containing the multiplicities
$e(n,m)$, which readily leads to

\begin{theorem}
\label{pinnervaalerthm1}
Let $\beta$ be a Parry number for which the Parry polynomial 
$n_{\beta}^{*}(X)$
factors into irreducible polynomials in 
$\mathbb{Q}[X]$ as in \eqref{numeratParryPol} and \eqref{producyclo}.
Then
\begin{itemize}
\item[(i)] for every $\epsilon > 0$ and 
$\mathcal{R} ~\geq~ \mathcal{R}_{0}(\epsilon)$,
\begin{equation}
\sum_{n=1}^{\infty} \frac{1}{\varphi(n)} \sum_{m=1}^{\varphi(n)} e(n,m)
~\leq~ (1+\epsilon) \, d_P \, \left(
\frac{\alpha(\mathbb{Q}) \, {\rm Log} \mathcal{R}}{\mathcal{R}}\right)^{1/2}, 
\end{equation}
\item[(ii)] for every $\epsilon > 0$ and
$\mathcal{R} ~\geq~ \mathcal{R}_{1}(\epsilon)$, 
\begin{equation}
\sum_{n=1}^{\infty} \sum_{m=1}^{\varphi(n)} e(n,m)
~\leq~ (1+\epsilon) \, d_P \, \left(
\frac{\alpha(\mathbb{Q}) \, {\rm Log} \mathcal{R}}{\mathcal{R}}\right)^{1/2},
\end{equation}
\item[(iii)] for each positive integer $n \leq \mathcal{R}$,  
\begin{equation}
\sum_{m | n} \sum_{h=1}^{\varphi(m)} e(m,h)
~\ll~
d_P \, \left( \frac{n}{\mathcal{R}} \right)^{1/2} 
\end{equation}
\item[(iv)]
for each integer $n$ such that ~$\varphi(n)\leq \mathcal{R}$,
\begin{equation}
\sum_{m=1}^{\varphi(n)} e(n,m) ~\ll~ d_P \, 
\left(\frac{\varphi(n)}{\mathcal{R}}\right)^{1/2}  
\left\{
1 + \left(
\frac{{\rm Log}{\rm Log} 20 n}{{\rm Log}\left(
\frac{\mathcal{R} \, {\rm Log}{\rm Log} 20 n}{\varphi(n)}\right)}
\right)^{1/2}
\right\}
\end{equation}
\end{itemize}
\end{theorem}

Let us investigate the role played by the gappiness (lacunarity) 
of the R\'enyi $\beta$-expansion of unity on the number $s$ of
distinct cyclotomic factors in $n_{\beta}^{*}(X)$.
With the notations of \eqref{producyclo} we have:
$$s = \sum_{n=1}^{\infty} \sum_{m=1}^{\varphi(n)} \min\{1, e(n,m)\}.$$  
Denote by $\tau(n)$ the number of positive divisors of $n$
and define $\pi(m) := \# \{ \mbox{prime number}~ p \mid p \leq m\}$. 
We now refer the reader to \eqref{nonsimplepoly1}, 
\eqref{nonsimplepoly2} or \eqref{simplepoly3} for the
Parry polynomial of $\beta$ as a sum of monomials. 
Define $N_{m}(n_{\beta}^{*}) := $ the number of monomials in this sum.
The remarkable Theorem 1 and Theorem 2 (i) in Pinner and Vaaler
\cite{pinnervaaler2} yield
 
\begin{theorem}
\label{pinnervaalerthmII} 
Let $\beta$ be a Parry number for which the Parry polynomial
$n_{\beta}^{*}(X)$
factors into irreducible polynomials in
$\mathbb{Q}[X]$ as in \eqref{numeratParryPol} and \eqref{producyclo}.
If $$n_{\beta}^{*}(X) = \sum_{i=1}^{N_{m}(n_{\beta}^{*})} a_{i} X^{n_i},
\qquad \mbox{with}~ a_i \neq 0, $$
then the number $s$ of distinct cyclotomic factors 
of $n_{\beta}^{*}(X)$ satisfies
\begin{equation}
\label{hh1}
(i)\quad \mbox{for every}~\epsilon > 0,
\qquad  s ~\ll_{\epsilon}~ \bigl(d_P\bigr)^{\epsilon} \, N_{m}(n_{\beta}^{*}),  
\qquad \qquad \qquad \mbox{}
\end{equation}
\begin{equation}
\label{hh2}
(ii)\quad s ~\leq~ \inf 
\left\{ 
\left(
\sum_{i=1}^{N_{m}(n_{\beta}^{*})} \tau(n_i -n_j)
\right) \, 
2^{\pi(N_{m}(n_{\beta}^{*}))}
\mid ~j \, \in \, \{1, \ldots, N_{m}(n_{\beta}^{*})\}  
\right\}.  
\end{equation}
\end{theorem}

Theorem \ref{pinnervaalerthmII} (i) improves Theorem
\ref{schinzelthm}: it introduces in the upper bound
\eqref{hh1} the term $N_{m}(n_{\beta}^{*})$ for which the
quantity $d_P + 1 - N_{m}(n_{\beta}^{*})$ is an estimate of the gappiness
of $d_{\beta}(1)$ by
\eqref{nonsimplepoly1} and \eqref{nonsimplepoly2}, and,
by \eqref{simplepoly3}, of some possible identifications between
the digits $t_j$. 
Theorem \ref{pinnervaalerthmII} (ii) 
bears an ingredient which does not seem
to have been observed for Parry polynomials yet: 
{\em the number of monomials in
$n_{\beta}^{*}(X)$}. Indeed,
when it is small, 
the exponent $\pi(N_{m}(n_{\beta}^{*}))$ of 2
in the upper bound in \eqref{hh2} is small, and
this may imply a small number of distinct cyclotomic factors by \eqref{hh2}.  

\subsection{Non-cyclotomic factors}
\label{S3.4}  

This section deals with the beta-conjugates $\chi$ of $\beta$ 
for which the minimal polynomial (of $\chi$) is non-cyclotomic.
These minimal polynomials are
irreducible factors in the factorization of the
Parry polynomial of $\beta$: they are either a $g_{j}(X)$ or
a $\kappa_{j}(X)$ in \eqref{numeratParryPol}. In some
cases, when a beta-conjugate $\chi$ lies, together with its Galois conjugates,
very near the unit circle, then the form of the minimal polynomial 
of $\chi$ can be specified. The remarkable 
Theorems 1 and 2 in Cassels \cite{cassels} imply

\begin{theorem}
\label{casselsthm1}
If $\chi$ is a beta-conjugate
of a Parry number $\beta$ such that the minimal polynomial
$g(X)$ of $\chi$ is non-reciprocal, with
$n = \deg(g)$, if
$\chi_1, \ldots, \chi_{n - 1}$   
denote the Galois conjugates of $\chi = \chi_0$ 
(which are also beta-conjugates of $\beta$), then
either

\begin{itemize}
\item[(i)] $\displaystyle |\chi_j| ~>~ 1 + \frac{0.1}{n}$ for at least one $j \in \{0, 1, \ldots, n-1\}$, or
\item[(ii)] $g(X) = - g^{*}(X)$ if
$\displaystyle |\chi_j| ~\leq~ 1 + \frac{0.1}{n}$ holds for all $j=0, 1, \ldots, n-1$.
\end{itemize}
\end{theorem} 

In the second case,
since $g(X) = \prod_{j=0}^{n-1} (X - \chi_j) = 
- \prod_{j=0}^{n-1} (1 - \chi_j X)$ is monic, 
all the beta-conjugates 
$\chi_j$ of $\beta$ ($j=0, 1, \ldots, n-1$)
are algebraic units, i.e. $|N(\chi_j)| = 1$.

\begin{theorem}
\label{casselsthm2}
If $\chi$ is a beta-conjugate
of a Parry number $\beta$ such that the minimal polynomial
(of degree $n$) of $\chi$ is non-cyclotomic 
and where
$\chi_1, \ldots, \chi_{n - 1}$ denote the 
Galois conjugates of $\chi$ ($ = \chi_0$), if
\begin{equation}
\label{casselsthm2ineq}
|\chi_j| ~\leq~ 1 + \frac{0.1}{n^2} \qquad ~\mbox{for}~ j=0, 1, \ldots, n-1,  
\end{equation}
then
at least one of the beta-conjugates 
$\chi_0, \chi_1, \ldots, \chi_{n - 1}$
of $\beta$ has absolute value~$1$.
\end{theorem}

Theorem \ref{casselsthm1} and Theorem \ref{casselsthm2}
are often applicable because beta-conjugates of
Parry numbers are highly concentrated near the unit circle.  

\begin{theorem}
\label{betaconjuoutside1}
Let $\beta$ be a Parry number with Parry polynomial 
$n_{\beta}^{*}(X)$ factored as in \eqref{numeratParryPol}. Then all
its non-cyclotomic irreducible factors 
$\kappa_{j}(X)$ (with $j= 1, \ldots, q$)
and $g_{j}(X)$ (with $j=1, \ldots, u$) have at least one root of modulus $> 1$.
\end{theorem}

\begin{proof}
By Kronecker's theorem \cite{kronecker}, 
if $\chi$ is a beta-conjugate of $\beta$ which lies in
the closed unit disk with all its Galois conjugates, then it would be 
a root of unity, i.e. a root of one of the cyclotomic factors 
$\Phi_{n_j}(X)$ in
\eqref{numeratParryPol}, and never a root of one of the irreducible factors
$\kappa_{j}(X)$ or $g_{j}(X)$. Hence if we assume the existence of non-reciprocal irreducible factors
and of reciprocal non-cyclotomic factors in the factorization of the Parry polynomial of $\beta$, then
these factors possess the mentioned property,
hence the claim.
\end{proof}

\begin{corollary}
\label{betaconjucoutingmini}
Let $\beta$ be a Parry number with Parry polynomial
$n_{\beta}^{*}(X)$ factored as in \eqref{numeratParryPol}.
Then
$$\# \{ \mbox{beta-conjugate} ~\chi~ \mbox{of}~ \beta \mid |\chi| > 1 \}$$
(counted with multiplicities) is
$$
\left\{
\begin{array}{lll}
\geq & \sum_{j=1}^{q} \gamma_j + \sum_{j=1}^{u} \delta_j & \mbox{if}~ \quad q+u > 0,\\
= & 0 & \mbox{if}~ \quad q = u = 0.
\end{array}
\right.
$$
\end{corollary}

The computation of an upper bound of 
Mignotte's discrepancy function on the
non-cyclotomic irreducible factors of the Parry
polynomial of $\beta$
will be reported elsewhere.  

\subsection{Real positive conjugates of a Parry number}
\label{S3.5}

For $R(X) \in \mathbb{R}[X]$ let ${\rm rp}(R) :=$ be the 
number of real positive roots of
$R(X)$ (counted with multiplicities).  

\begin{proposition}
\label{realpositiveorigin}
Let $\beta$ be a Parry number.
We have: ${\rm rp}(n_{\beta}^{*}) \geq 1$,
and, if $\chi \in (0,1)$
is a Galois- or a beta-conjugate of $\beta$, then either
\begin{itemize}
\item[(i)] ~$P_{\beta}(\chi) = 0$, and if $P_{\beta}(X)$
is reciprocal then $\chi = \beta^{-1}$ is the only real positive Galois conjugate
of $\beta$, or
\item[(ii)] ~$\kappa_{j}(\chi)=0$ for some $j \in \{1, 2, \ldots, q\}$ where
$\kappa_{j}(X)$ is one of the non-reciprocal irreducible 
factors in the factorization \eqref{numeratParryPol} of
$n_{\beta}^{*}(X)$.       
\end{itemize}
\end{proposition}

\begin{proof}
The reciprocal irreducible factors in
\eqref{numeratParryPol} which are cyclotomic polynomials have no root
outside $|z|=1$, therefore cannot cancel at $\chi$. 
Assume now that an irreducible factor in \eqref{numeratParryPol}
is reciprocal, non-cyclotomic, and cancels at 
$\chi \in (0,1)$. Let us show that it is impossible. 
Indeed, it would also have
$\chi^{-1}$ as conjugate root of $\chi$, hence $\chi$
and $\chi^{-1}$ would be beta-conjugates of $\beta$.
Since $\Omega \cap (1,+\infty) = \emptyset$ and that $\Omega$ 
contains all beta-conjugates of all Parry numbers,  
it would implies the existence of a beta-conjugate of
$\beta$ outside $\Omega$, which is a contradiction. 
\end{proof}

\begin{proposition}
\label{mignottelowerbound}
Let $\beta$ be a Parry number.
Then
\begin{equation}
\label{mignottelowerineq}
\left|
\frac{{\rm rp}(n_{\beta}^{*})}{d_P}
\right|^{2}
~\leq~ \frac{2 \pi}{k} \cdot {\rm dis}(n_{\beta}^{*}).
\end{equation}
\end{proposition}

\begin{proof}
Let us consider
the angular sector $0 \leq \alpha \leq 2 \pi - \alpha \leq 2 \pi$ with
$\alpha > 0$ small enough so that its complementary sector 
only contains $\beta$ and the real positive 
conjugates of $\beta$.   
From \eqref{ineqdis}, with $\alpha \to 0^{+}$, 
we deduce the lower bound \eqref{mignottelowerineq}
of Mignotte's discrepancy function.
\end{proof}

There are two natural questions: 
\begin{itemize}
\item[(i)] 
given $n \geq 1$ an integer
what is the average value of
${\rm rp}(n_{\beta}^{*})
$ over all Parry polynomials
$n_{\beta}^{*}$ of degree $d_P = n$?
Denote by $\mathbb{E}_{P}(n)$ this average value;  
\item[(ii)]
how behaves 
$\mathbb{E}_{P}(n)/n$ when $n$ tends to infinity?  
\end{itemize}
 
The general context of such questions is Kac's formula
\cite{kac} and its recent improvements \cite{edelmankostlan}. Let us recall 
it first. The expected 
number of real zeros $\mathbb{E}(n)$ of a random polynomial of degree $n$
with real coefficients
is given by  
\begin{equation}
\label{kacasympto}
\mathbb{E}(n) = \frac{2}{\pi} \log (n) + 0.6257358072\ldots + \frac{2}{n \pi} + O(1/n^2)
\qquad \mbox{for}~ n \to \infty,
\end{equation}
assuming coefficients are following independent standard normal laws. 
The $\frac{2}{\pi} \log n$ term was obtained by Kac in 1943
and the other terms require integral formulas 
and their asymptotic series 
from curves drawn on spheres \cite{edelmankostlan}.   
For each $n \geq 1$ these averaging techniques 
could be adapted to the smaller set of Parry polynomials of degree $n$, instead of 
the whole set of polynomials of $\mathbb{R}[X]$ of degree $n$, and
to real positive roots only, to compute
$\mathbb{E}_{P}(n)$, but
the set $\mathcal{P}\mathcal{P}$ is badly described 
and is not suitable for this type of computation.
At least, since $\lim_{n \to \infty} \mathbb{E}(n)/ n = 0$
by \eqref{kacasympto} we could expect 
\begin{equation}
\label{expectationn}
\lim_{n \to \infty} \mathbb{E}_{P}(n)/n = 0.
\end{equation}  

It seems that we cannot prove \eqref{expectationn}
yet.
However, the remarkable Theorem 4.1 
of Borwein, Erd\'elyi and K\'os \cite{borweinerdelyikos}
readily implies, for a 
large class of Parry numbers $\beta \in (1, 2)$
which are algebraic units,  

\begin{theorem}
\label{borweineerdelyikosthm}
There is an absolute constant $C_3 > 0$ such that, for all Parry numbers $\beta$ such that
$${\rm H}(n_{\beta}^{*}) = 1 \qquad \mbox{and} \qquad |N(\beta)| = 1,$$
the inequality
\begin{equation}
\label{borweinerdelyikosineq}
\frac{{\rm rp}(n_{\beta}^{*})}{d_P}   ~\leq~ C_3 \, \frac{1}{\sqrt{d_P}}
\end{equation}
holds.
\end{theorem}

When the Parry polynomial is irreducible,
the remarkable Theorem in Mignotte \cite{mignotte1} p. 83
leads to a result of the same type: it
readily implies

\begin{theorem}
\label{mignotterpthm}
For every Parry number $\beta$ such that
the Parry polynomial $n_{\beta}^{*}(X)$
is irreducible, then
the inequality
\begin{equation}
\label{borweinerdelyikosineq}
\frac{{\rm rp}(n_{\beta}^{*})}{d_P}   ~\leq~
\sqrt{\frac{2 \pi}{k}} \, \sqrt{3 \, {\rm Log} (2 d_P) + 
4 \, {\rm Log} \, {\rm M}(P_{\beta})} \, \frac{1}{\sqrt{d_P}}
\end{equation}
holds.
\end{theorem}

Let us now show that 
Mignotte's Theorem \ref{mignottethm} 
gives in a simple way an upper bound to the multiplicity
of a beta-conjugate of a Parry number,  
valid for any beta-conjugate, complementing then  
Theorem \ref{multiplicitybetaconju}, 
Corollary \ref{betaconj13simple} and Theorem
\ref{pinnervaalerthm1}.

Let $\beta$ be a Parry number with Parry polynomial
$n_{\beta}^{*}(X)$ of degree $d_P$.
Then the multiplicity $q$ of a beta-conjugate of
$\beta$ satisfies:
\begin{equation}
\label{betaconjumultineq}
q ~\leq~ \sqrt{\frac{2 \pi}{k}} \sqrt{\mbox{dis}(n_{\beta}^{*})} \, d_P \,, 
\end{equation}
where $k$ is Catalan's constant.

This inequality \eqref{betaconjumultineq}
is obtained as follows: 
let $q$ be the multiplicity of a beta-conjugate
$\chi \in \Omega$ of $\beta$ in the Parry polynomial
$n_{\beta}^{*}(X)$. Let us write
$\chi = r e^{i \phi}$ with $r > 0$. Let us take
$\eta = \phi + \epsilon/2$, $\alpha = \phi -
\epsilon/2$, for $\epsilon > 0$ small enough so that
(i) the angular sector
$$\mathcal{S}_{\phi, \epsilon} := \{z \mid \mbox{arg}(z) \in [\alpha, \eta]\}$$
contains
$\chi$, with eventually other beta-conjugates or Galois-conjugates having the same argument
$\phi$, but does not contain other roots of
$n_{\beta}^{*}(X)$ having an argument
$\neq \phi$,
(ii) 
the angular sector $\mathcal{S}_{0,\epsilon}$ only contains the real positive conjugates
of $\beta$, including $\beta$. By rotating
$\mathcal{S}_{0, \epsilon}$ of an angle $\phi$
and allowing $\epsilon$ to tend to $0^{+}$, we obtain
from Mignotte's Theorem \ref{mignottethm}
$$\left|\frac{q}{d_P}\right|^{2}  
~\leq~ \frac{2 \pi}{k} \, \mbox{dis}(n_{\beta}^{*}),$$
from which \eqref{betaconjumultineq} is deduced.
This proof 
uses the relation 
$e^{i \phi} \mathcal{S}_{0, \epsilon} = 
\mathcal{S}_{\phi,\epsilon}$ with $\epsilon$ very small, 
and the counting
processes of the roots of the Parry polynomial at $0$ and at $\phi$. 
They are correlated: a large number of 
real positive conjugates of $\beta$ 
leads to large Mignotte's discrepancies
as the inequality 
\eqref{mignottelowerineq} shows it, and this means that 
the {\it common} upper bound 
\eqref{betaconjumultineq}
of the multiplicities of the beta-conjugates 
is probably
not very good in this case.
This is likely to occur when
the number of non-reciprocal irreducible factors
in the factorization of the Parry polynomial of $\beta$
is large, from Proposition \ref{realpositiveorigin} (ii).
However, this type of upper bound is not good
from a numerical viewpoint as shown in Section \ref{S4},
what suggests that Mignotte's approach could be improved.
 
\subsection{An equidistribution limit theorem}
\label{S3.6}

Theorem \ref{landauclustering}, Theorem \ref{mignottethm} 
and Conditions \eqref{seqinfinidegrees} 
express the ``speed of convergence" and the 
``angular equidistributed character"
of the conjugates of a Parry number, towards the unit circle,
or of the collection of Galois conjugates 
and beta-conjugates of a sequence of Parry numbers.
So far, the limit of this equidistribution phenomenon is not yet formulated.
In which terms should it be done? What is the natural framework 
for considering at the same time
all the conjugates of a Parry number and what is the topology
which can be invoked?  

In this respect we will follow Bilu's equidistibution limit
theorem in Bilu \cite{bilu} \cite{granville}, though the conditions 
of convergence of Parry numbers are here much 
more general than those considered
by Bilu.  

Let $\beta$ be a Parry number for which all beta-conjugates are
simple roots of the Parry polynomial
$n_{\beta}^{*}(X)$. Let $\mathbb{K}$ be the algebraic number field
generated by $\beta$, its Galois conjugates and its beta-conjugates
over $\mathbb{Q}$.
We have the following field extension: 
$\mathbb{K} \supset \mathbb{Q}(\beta)$ and $\mathbb{K}/\mathbb{Q}$ is
Galois.
We denote by $\mathbb{K}_v$ the completion of
$\mathbb{K}$ for the Archimedean or non-Archimedean place $v$ of
$\mathbb{K}$. 
The absolute logarithmic height of $\beta$ is defined as:
$$h(\beta) := \frac{1}{[\mathbb{K}:\mathbb{Q}]}\,
\sum_v [\mathbb{K}_v : \mathbb{Q}_v ]\, \max\{0,{\rm Log}|\beta|_v\}.$$
Let us now consider the whole set of Galois conjugates $(\beta^{(i)})$
and
beta-conjugates $(\xi_j)$ of $\beta$. Denote 
$$\Delta_{\beta} := \frac{1}{[\mathbb{K} : \mathbb{Q}]}
\sum_{\sigma: \mathbb{K} \to \mathbb{C}} \, \delta_{\{\sigma(\beta)\}}$$
the weighted sum of the Dirac measures at all the conjugates 
$\sigma(\beta)$ of $\beta$, where $\sigma$ runs over the $d_P$
$\mathbb{Q}$-automorphisms of 
$\mathbb{K}$
$$\sigma: \beta \to \beta^{(i)}, \qquad \mbox{or}\qquad
\sigma: \beta \to \xi_j$$
which send $\beta$ either to one of its Galois conjugates 
or to one of its beta-conjugates.

Let us recall that a sequence $(\alpha_k)$ of points in
$\overline{\mathbb{Q}}^{*}$ is strict if any proper algebraic subgroup of
$\overline{\mathbb{Q}}^{*}$ contains $\alpha_k$ for only finitely many values of $k$.

The topology which is used is the following: a sequence of probability measures
$\{\mu_k\}$ 
on a metric space S weakly converges to $\mu$ if for any bounded
continuous function $f: S \to \mathbb{R}$ we have
$$(f,\mu_k) \to (f, \mu) \qquad \mbox{as}~ k \to +\infty.$$
The remarkable Theorem 1.1 in \cite{bilu}
readily implies

\begin{theorem}
\label{biluthm}
Let $(\beta_i)_{i \geq 1}$ be a strict sequence of Parry numbers whose Parry polynomials
have all simple roots, and which satisfies   
\begin{equation}
\label{condbilu}
\lim_{i \to \infty} h(\beta_i) ~=~ 0.
\end{equation}
Then
\begin{equation}
\lim_{i \to \infty} \Delta_{\beta_i} ~=~ \nu_{\{|z|=1\}}
\qquad \quad \mbox{weakly}
\end{equation}
where
$\nu_{\{|z|=1\}}$ is the Haar measure on the unit circle.
\end{theorem}

In the proof of his theorem Bilu uses the Erd\H{o}s-Tur\'an theorem
as basic ingredient and the fact that the minimal polynomials of the $\beta_i$s'  
have distinct roots. Here
working with non-irreducible Parry polynomials 
for which all the roots are distinct and Mignotte's theorem
suffices
to give the same conclusion.

Let us observe that the convergence condition \eqref{condbilu}  
means in particular that $\beta_i \to 1^{+}, i \to \infty$, while
convergence conditions \eqref{seqinfinidegrees} cover many other cases 
for general sequences of Parry numbers.
In the convergence condition \eqref{condbilu} is also
included some conditions on the $p$-dic valuations of the beta-conjugates
of the $\beta_i$s'. These aspects will be reconsidered elsewhere by the author.

\section{Examples}
\label{S4}

Table 1 gives Mignotte's discrepancy function
$\frac{2 \pi}{k} \frac{\tilde{h}(n_{\beta}^{*})}{d_P}$
($k$ is Catalan's constant) relative to the following four Pisot numbers:
\begin{itemize}
\item the confluent Parry number $\beta=9.999\ldots$ dominant 
root of $X^{40} - 9 \sum_{i=1}^{39} X^i - 4 = P_{\beta}(X) 
= n_{\beta}^{*}(X)$. It is 
a Pisot number which has no 
beta-conjugate \cite{vergergaugry2}, for which
$d_{\beta}(1) = 0. k_{1}^{d-1} k_2$, with $k_1=9, k_2=4$
and
$d= d_P = 40$. The height of the Parry polynomial
of $\beta$ is 9,
\item Bassino's cubic Pisot number $\beta=30.0356\ldots$ dominant root of
$X^3 -(k+2) X^2 + 2k X - k = P_{\beta}(X)$, with $k=30$,
for which the complementary factor 
is the product 
$\Phi_2 \Phi_3 \Phi_5 \Phi_6 \Phi_{10} \Phi_{15} \Phi_{30} \Phi_{31}$
of cyclotomic factors.  
The height of the Parry polynomial of $\beta$ is 30 and $d_P= 62$,
\item the smallest Pisot number $\beta = 1.767...$ for which the 
complementary factor is (NC) reciprocal 
and non-cyclotomic (Boyd \cite{boyd2} p. 850): 
it is the dominant root of
$P_{\beta}(X) = 
X^{12} - X^{10} - 2 X^9 - 2 X^8 - X^7 - X^6 - X^5 - X^4 + X^2 + X + 1$
and has
$\Phi_4 \Phi_6 \Phi_{12} \Phi_{30} L(-X)$
as complementary factor where 
$L(X)=X^{10} + X^9 - X^7 -X^6 - X^5 -X^4 -X^3 +X+1$ is Lehmer's polynomial.  
The R\'enyi $\beta$-expansion of 1
has preperiod length 4 and period length 34. The Parry polynomial 
of $\beta$ has degree $d_P$ equal to 38 and height 1, 
\item the second-smallest Pisot number $\beta=1.764\ldots$
for which the complementary factor is (NR) non-reciprocal
(Boyd \cite{boyd2} p. 850). We have
$P_{\beta}(X) =X^{11} - 2 X^9 - 2 X^8 -X^7 + 2 X^5 + 2 X^4 + X^3 -X -1$ and
$\Phi_6 G(X)$ as complementary factor, where
$G(X)= X^{22} + X^{15} +X^8 -X^7 -1$ is non-reciprocal.  
The R\'enyi $\beta$-expansion of 1
has preperiod length 30, period length 5, and
H$(n_{\beta}^{*}) = 1$.
\end{itemize} 
On each line, in the column ``Parry", 
is reported Mignotte's discrepancy function
with
the value (ET) of the discrepancy function
$16^{2} \times \frac{1}{d_P}
\mbox{Log}\left(\frac{\|n_{\beta}^{*}\|_1}{\sqrt{|n_{\beta}^{*}(0)|}}\right)$
of Erd\H{o}s-Tur\'an 
for comparison.

\begin{figure}
\begin{center}
\includegraphics[width=8cm]{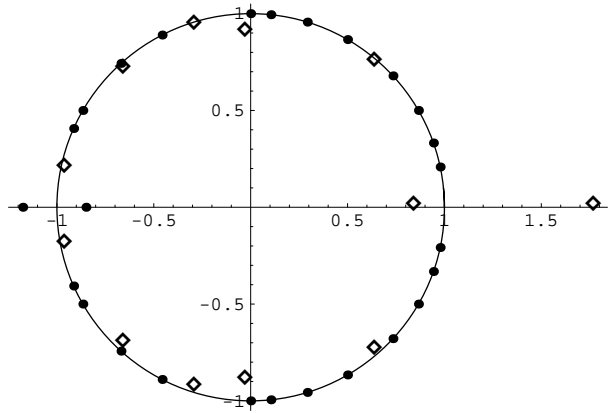}
\end{center}
\caption{Concentration and equi-distribution of the Galois
conjugates ($\diamond$) $\neq \beta$
and the beta-conjugates ($\bullet$)
of the smallest NC Pisot number
$\beta = 1.767\ldots$ near the unit circle.}
\label{smallestNCPisot}
\end{figure}

In the column ``Mini" is reported the value
$\frac{2 \pi}{k} \frac{\tilde{h}(P_{\beta})}{d_P},$
resp. in the column ``cycl." the value
$\frac{2 \pi}{k} \frac{\tilde{h}(\prod_{j=1}^{s} \Phi_{n_j}^{c_j})}{d_P},$
resp. in the column ``rec. non-cycl." the value
$\frac{2 \pi}{k} \frac{\tilde{h}(\prod_{j=1}^{q} g_{j}^{\delta_j}
)}{d_P},$
resp. in the column ``non-rec." the value
$\frac{2 \pi}{k} \frac{\tilde{h}(\prod_{j=1}^{u} \kappa_{j}^{\gamma_j})}{d_P},$
with the notations of \eqref{numeratParryPol}.

The sharpness of the splitting 
\eqref{decompodis} is not too bad 
in these examples: the ratio between 
the sum of Mignotte's discrepancies of the irreducible factors
of the Parry polynomial 
and Mignotte's discrepancy applied to the Parry polynomial itself is
always less than 4, and the sum of Mignotte's 
discrepancies of the factors is always much lower than the value ET.

\vspace{0.4cm}

\noindent
\begin{tabular}{c||c|c|c|c|c|}
$\beta$ & ``Parry" & ``Mini." & ``cycl." & ``rec.  & ``non-rec."
\\
 & & & & non-cycl." & \\
\hline
Confluent & $0.0695\ldots$ & $0.0695\ldots$ & / & / & / \\
$k_1 = 9, k_2 =4$ & (ET$=33.16\ldots$)& & & & \\\hline
Bassino & $0.0631\ldots$ & $0.106\ldots$ & $0.0893\ldots$ & / & / \\
$k= 30$ & (ET$=21.20\ldots$)& & & & \\
\hline
$1.767\ldots$ smallest & $0.0979\ldots$ & $0.0927\ldots$ & $0.0946\ldots$ & $0.100\ldots$ & \\
NC Pisot & (ET$=17.77\ldots$)& & & & \\
\hline
$1.764\ldots$ second-smallest& $0.107\ldots$ & $0.124\ldots$ & $0.0761\ldots$ & / & $0.0840\ldots$ \\
NR Pisot & (ET$=22.26\ldots$)& & & & \\
\hline
\end{tabular}

\vspace{0.2cm}

\begin{center}
Table 1.
\end{center}

\vspace{0.2cm}

Ganelius, Mignotte and Amoroso \cite{amoroso1}
\cite{amoroso2} \cite{ganelius} \cite{mignotte2}
have already mentioned the (numerically) bad discrepancy function given by
Erd\H{o}s-Tur\'an and Table 1 shows it as well: there
exists a factor greater than 180 between ET and 
Mignotte's discrepancy applied to the Parry polynomial, even
much larger in some other cases.
The upper bound of the multiplicities of the beta-conjugates,
computed from Mignotte's discrepancy function, according to
\eqref{betaconjumultineq}, for the four cases of Table 1, 
is respectively: 27, 40, 31, 30. These values are 
much higher than the true one: 1 in each case.

\begin{figure}
\begin{center}
\includegraphics[width=8cm]{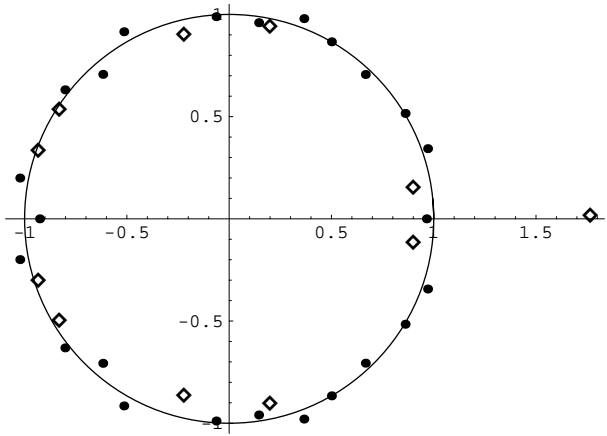}
\end{center}
\caption{Concentration and equi-distribution of the Galois
conjugates ($\diamond$) $\neq \beta$
and the beta-conjugates ($\bullet$)
of the second-smallest NR Pisot number
$\beta = 1.764\ldots$ near the unit circle.}
\label{secondsmallestNRPisot}
\end{figure}

Figure 3 and Figure 6 in Verger-Gaugry
\cite{vergergaugry2} show the equi-distribution of
the conjugates of the first two examples of Pisot numbers
near the unit circle, and are not reported here.
The last two cases are illustrated in Figure
\ref{smallestNCPisot} and Figure \ref{secondsmallestNRPisot}.

\section{Arithmetics of Perron numbers and non-Parry case}
\label{S5}

If $\beta$ is a Perron number 
which is not a Parry number the analytical function
$f_{\beta}(z)$ has the unit circle as natural boundary, 
by Szeg\H{o}'s Theorem \ref{szegothm}.
It is such that $f_{\beta}(1/\beta) = 0$ and satisfies
$$\overline{f_{\beta}(D(0,1))} ~=~ \mathbb{C}$$
by Theorem 1 in Salem (\cite{salem} p. 161).
Questions on the number and the type of the other zeros 
of this analytical function 
in the open unit disk, in particular beta-conjugates, 
can be found in \cite{vergergaugry2}.  

The set of Perron numbers
$\mathbb{P}$ admits an arithmetic structure which
does not take into account whether a 
Perron number is a Parry number or not (Lind, 
Section 5 in \cite{lind}).
First let us recall two theorems.

\begin{theorem}
\label{pisotcorpsthm}
Every real algebraic number field $\mathbb{K}$ 
is generable by 
a Pisot number. If $d=[\mathbb{K} : \mathbb{Q}]$ denotes the degree
of the field extension $\mathbb{K} / \mathbb{Q}$, the 
number field $\mathbb{K}$ contains infinitely many
Pisot numbers of degree $d$, some of which being algebraic units. 
\end{theorem}

\begin{proof}
It is a tradition to call Pisot numbers $S$-numbers
(\cite{bertinetal} p. 84). This is Theorem 5.2.2 in
\cite{bertinetal}. 
\end{proof}

Let $S$ be the set of Pisot numbers and 
$T$ the set of Salem numbers (\cite{bertinetal} p. 84).
We have: $S \subset \mathbb{P}_P$ \cite{bertrandmathis} \cite{schmidt}  
and
$$T \cap \mathbb{P}_P ~\neq~ \emptyset, \qquad 
T \cap \mathbb{P}_a ~\neq~ \emptyset.$$
This dichotomy of Salem numbers is still obscure.

\begin{theorem}
\label{salemcorpsthm}
Let $\beta \in T$. The algebraic 
number field $\mathbb{K} = \mathbb{Q}(\beta)$ is a real quadratic extension of a totally real field.
There exists $\tau_0 \in \mathbb{K} \cap T$ such that
\begin{itemize}
\item[(i)] $\mathbb{K} = \mathbb{Q}(\tau_0)$, 
\item[(ii)]  
$\mathbb{K} \cap T ~=~ \{\tau_{0}^{n} \mid n =1, 2, \ldots \}$.
\end{itemize}
Every number in
$\mathbb{K} \cap T$ is quotient of two numbers in $\mathbb{K} \cap S$. 
\end{theorem}

\begin{proof}
Theorem 5.2.3 in
\cite{bertinetal}, or \cite{lalande}.
\end{proof}

Lind has introduced the notion of irreducible Perron number in
\cite{lind} Section 5, based on the fact that 
$\mathbb{P}$ is closed under multiplication and addition
(\cite{lind} Proposition 1) 
and on the following (\cite{lind} Proposition 5)  

\begin{proposition}
\label{lindpropo}
If $\lambda = \alpha \beta$ with $\lambda, \alpha, \beta \in \mathbb{P}$
then 
$\alpha, \beta \in \mathbb{Q}(\lambda)$.  
\end{proposition}

Let us observe that $1 \in \mathbb{P}$ (by convention)
and that $1 \not\in S, 1 \not\in T$.

\begin{definition}
A Perron number $\lambda \in \mathbb{P}$ is said irreducible if $\lambda > 1$
and if it cannot be written as $\alpha \beta$ with
$\alpha, \beta \in \mathbb{P}$ and $\alpha, \beta > 1$. 
\end{definition}

\begin{theorem}
\label{lindcorpsthm}
Every Perron number $\beta > 1$ 
can be factored into a finite number of
irreducible Perron numbers $\lambda_i$ :
$$\beta ~=~ \lambda_1 \lambda_2 \ldots \lambda_s.$$  
There is only a finite number of such factorizations of $\beta$,
and unique factorization of $\beta$ may occur
(two factorizations of $\beta$ are the same when they differ
only by the order of the terms).
\end{theorem}

\begin{proof}
Theorem 4 in \cite{lind}.
\end{proof}

In Theorem \ref{lindcorpsthm} the irreducible
Perron numbers $\lambda_i$
all belong to the number field 
$\mathbb{Q}(\beta) \subset \mathbb{R}$ by Proposition
\ref{lindpropo}.
The existence of non-unique factorizations in a number field
is implied by the following remarkable 
theorem of Lind (\cite{lind} Theorem 5).  

\begin{theorem}
\label{lindNONUNIthm}
Let $\mathbb{K}$ be an algebraic number field. 
The following are equivalent:
\begin{itemize}
\item[(i)] $\mathbb{K} \cap \mathbb{P}$ contains elements which 
have non-unique factorizations,
\item[(ii)] $\mathbb{K} \cap \mathbb{P}$ contains non-rational Perron numbers,
\item[(iii)] $\mathbb{K} \cap \mathbb{R} ~\neq~ \mathbb{Q}$.
\end{itemize} 
\end{theorem}

A basic question is
about the nature and the 
dispatching of the irreducible Perron numbers
in the two classes $\mathbb{P}_P$ and
$\mathbb{P}_a$.
The characterization of the family of
irreducible Perron numbers in 
a given number field is 
obscure, in particular when the number field is generated
by a non-Parry Perron number. 
By Theorem \ref{pisotcorpsthm} non-unique factorizations occur
in every real number field $\mathbb{K}$ since such number 
fields are generated by
Pisot numbers, but though Pisot numbers are always
Parry numbers, this does mean that the irreducible 
Perron numbers in $\mathbb{K}$ are necessarily Parry numbers.

\begin{corollary}
\label{salemNONUNIcoro}
For every Salem number $\beta$ the number field
$\mathbb{Q}(\beta)$ contains elements which
have non-unique factorizations into irreducible Perron numbers of
$\mathbb{Q}(\beta)$.
\end{corollary}

\begin{proof}
It is a consequence of Theorem
\ref{salemcorpsthm} (ii) and of
Theorem \ref{lindNONUNIthm} (iii).  
\end{proof}

\frenchspacing

\section*{Acknowledgements}

The author is indebted to Y. Bilu, L. Habsieger, P. Liardet  and G. Rhin  
for  valuable comments and discussions.

\end{document}